\documentclass[12pt]{article}

\usepackage{amsmath,amssymb,amsthm}
\usepackage{xcolor}
\usepackage{hyperref}
\usepackage{bm}
\usepackage{graphicx}
\usepackage{subcaption}
\usepackage{url}

\newtheorem{condition}{Condition}
\newtheorem{example}{Example}
\newtheorem{theorem}{Theorem}
\newtheorem{definition}{Definition}
\newtheorem{corollary}{Corollary}

\title{Scale Dilation Dynamics in Flexible Bandwidth Needlet Constructions}

\author{Claudio Durastanti\\
	Department of Basic Science and Applied to Engineering,\\
	Sapienza University of Rome, Italy\\
	\texttt{claudio.durastanti@uniroma1.it}
}

\date{} % empty date to omit printing

\begin{document}
	
	\maketitle

\begin{abstract}
Flexible bandwidth needlets offer a versatile multiscale framework for analyzing functions on the sphere. A key element in their construction is the dilation sequence, which controls how the multipole consecutive scales are spaced and overlapped. At any resolution level, this sequence determines the center positions of the needlet weight functions and influences their localization in the spatial domain and spectral concentration properties by means of the relative bandwidth ratio. In this paper, we explore the different asymptotic regimes that arise when the dilation sequence exhibits shrinking, stable (standard), or spreading behavior. Moreover, we assume the dilation sequence grows regularly enough to ensure well-defined asymptotic properties. For each regime, we characterize the impact on the geometry of the center scales and the shape of the multipole windows, with particular attention to their overlap structure and spectral coverage. These insights help to clarify the trade-offs between localization, redundancy, and scalability in the design of needlet-type systems, particularly in relation to the study of the asymptotic uncorrelation of needlet coefficients when applied to random fields.\\
\textbf{Keywords:} Flexible bandwidth needlets, Multiscale analysis, Dilation factor dynamics, Central scale growth, Needlet frame construction, Localization properties, Asymptotic uncorrelation.\\
\textbf{MSC2020:} 42C40, 60G60
\end{abstract}

%%%Research highlights
%\begin{highlights}
%	\item Impact of dilation sequences on needlet localization and spectral properties.
%	The paper systematically studies how different dilation sequences—shrinking, stable, or spreading—affect the spatial localization and spectral concentration of flexible bandwidth needlets, revealing their influence on center positions and relative bandwidth ratios at each scale.
%	\item Trade-offs in needlet design and asymptotic uncorrelation.
%	By characterizing the geometric and overlap structures of multipole windows under various dilation regimes, the work clarifies fundamental trade-offs between localization, redundancy, and scalability, with direct implications for the asymptotic uncorrelation behavior of needlet coefficients on random fields.
%\end{highlights}

%
\section{Introduction}
%We focus on the \emph{shrinking regime} In statistical applications—such as nonparametric regression or density estimation on $\mathbb{S}^2$. Indeed, shrinking needlets support adaptive thresholding that automatically adjusts to unknown smoothness, achieving near–minimax risk bounds. In this paper, we will employ shrinking needlets to establish a quantitative central limit theorem for needlet-based Poisson spherical waves, demonstrating how their variable bandwidth facilitates precise probabilistic approximations in high-resolution spherical signal processing.

Spherical needlets provide a powerful framework for the analysis of data on the sphere, particularly when observations are localized to subregions of the sphere. They are spherical wavelets constructed to achieve simultaneous localization in both real (spatial) and harmonic (frequency) domains, and they yield quasi-uncorrelated coefficients. Defined in their standard dyadic form in \cite{npw1,npw2}, needlets have been extensively employed in probability theory, mathematical statistics, and cosmological data analysis, where their localization and quasi-uncorrelated structure offer significant analytical advantages. Notable applications include: characterizations of Gaussian random fields \cite{bkmpAoS,cammar,shetod}, asymptotic behavior of statistics in Poisson fields \cite{boudur1}, nonparametric estimation theory \cite{bkmpAoSb,dur2}, Whittle likelihood inference on the sphere \cite{dlmejs}, point-source detection \cite{duque}, and statistical testing for isotropy and non-Gaussianity in the cosmic microwave background (CMB) \cite{LeGia}. In CMB analysis, they are especially effective for component separation \cite{zz,carones}, owing to their capacity to adapt across scales and regions.\\
Subsequent generalizations have broadened the scope of needlet constructions; see, for instance, \cite{needlib,gm2}. In particular, the recent framework introduced in \cite{dmt24} replaces the classical dyadic scaling with an arbitrary sequence of spectral scales (or centers) $\lbrace S_j \ :j \geq 1\rbrace$, indexed by the resolution level $j \in \mathbb{N}$ and controlling the harmonic support, so that each needlet is associated with the frequency band $\left[S_{j-1}, S_{j+1}\right]$. This allows for greater flexibility in tuning localization and scale resolution, a feature especially beneficial for analyzing data with variable resolution or anisotropic noise across the sphere. 
In the standard case, indeed, the same weight - or window - function is rescaled at each resolution level $j$, so the spectral shape remains identical across all levels. In contrast, flexible needlets allows this trade-off to evolve with $j$, letting the shape and extent of spectral localization to adapt as resolution increases. More in detail, for each $j$ the \emph{relative bandwidth ratio} as 
\begin{equation*}
\Delta_j=\frac{S_{j+1}-S_{j-1}}{S_j}
\end{equation*}
measures the proportion between the bandwidth of the spectral window and its center $S_j$, that is, how many multipoles are taken by the needlets with respect to their scale. In the standard case, $\Delta_j$ is a positive constant. In contrast, in the flexible bandwidth setting, $\Delta_j$ may tend to zero or diverge, allowing for finer control over the tradeoff between localization in harmonic space and on the sphere and then leading to the shrinking or spreading needlets.
This enhanced capability to tailor properly the concentration properties to any specific problem can be of help in several applications. From a more theoretical standpoint, particularly in adaptive nonparametric estimation on the sphere, the goal is to estimate an unknown regression function $f$ whose smoothness is not known in advance. Standard needlet constructions use a fixed bandwidth ratio across scales, which limits the ability to adapt to local variations in the function's regularity (see \cite{dur}). In contrast, flexible bandwidth needlets allow the bandwidth to vary with the resolution level $j$: it can be narrowed in certain multipole ranges to capture rapid oscillations, typically high-frequency features, and widened elsewhere to smooth out noise. This adaptability should enable the procedure to achieve nearly minimax convergence rates over a broad range of Sobolev or Besov spaces, without requiring prior knowledge of the underlying smoothness. This approach can be readily extended to the estimation of derivatives of arbitrary order and to a broad class of compact manifolds beyond the sphere (see for example \cite{durtur23}).
 Similarly, in proving central limit theorems for nonlinear functionals (e.g., power-spectrum or bispectrum‐based statistics) of a random field, reducing overlap between high-$\ell$ windows via narrower bands simplifies cumulant bounds and yields stronger rates of convergence for the asymptotic normality of needlet coefficients.
In practical CMB analysis, flexible bandwidth needlets improve component separation and feature detection by matching each filter to the observed data’s scale‐dependent characteristics. For example, ground‐based experiments like SPT (South Pole Telescope) or SO (Simons Observatory) exhibit noise and beam responses that vary strongly with $\ell$: by shrinking the bandwidth where noise dominates (e.g., $\ell >4000$, and widening it where the beam is sharp (e.g.,  $\ell \sim 8000-10000$), one can suppress noisy multipoles and preserve small‐scale features such as Sunyaev–Zel’dovich clusters (see, for example \cite{bleem}). This flexibility yields cleaner separation of astrophysical components and more robust detection of scale‐dependent signals than fixed‐bandwidth constructions.\\

In the original work \cite{dmt24}, the authors characterize the flexible bandwidth needlet projectors $\lbrace \Psi_{j}\left(\cdot,\cdot\right) : j \geq 1 \rbrace$ in terms of their window function sequences $\lbrace b_{j}\left(\cdot\right): j \geq 1\rbrace$, developing concentration bounds under suitable conditions on the scale sequence $\lbrace S_j : j \geq 1\rbrace$. In particular, we are referring to the spatial localization property:  for any $x,y \in \mathbb{S}^2$, $M\in \mathbb{N}$, 
\begin{equation*}
	\begin{split}
	\left \vert \Psi_{j} \left(x,y\right) \right \vert& \leq C_M \left(S_{j+1}^2-S_{j-1}^{2}\right)\\ &\times\max\left(\frac{1}{\left(S_{j-1}d_{\mathbb{S}^2}(x,y)\right)^{2M}},\frac{1}{\left(\left(S_{j}-S_{j-1}\right)d_{\mathbb{S}^2}(x,y)\right)^{2M}}\right),
	\end{split}
\end{equation*}
where $C_M>0$ and $d_{\mathbb{S}^2}(x,y)$ is the great circle distance between $x$ and $y$.
With regard to applications, the authors introduce a novel goodness-of-fit test based on the squared flexible bandwidth needlet coefficients computed over a spherical subdomain. They establish a quantitative central limit theorem for the test statistic using the Malliavin–Stein method. The validity of the test relies on two key conditions: a vanishing bandwidth regime ($\Delta_j\to0$) and a sufficiently strong one-step separation between scales ($S_j-S_{j-1}\gtrsim S_j^{1-\beta}$). Here, $\beta \in \left[ \left. 0,1\right)\right.$ is a smoothness parameter that governs the decay of the spectral modulation of the variance function, ensuring that correlations between coefficients decay rapidly enough to enable reliable statistical inference.\\

In this work, we adopt a different viewpoint. First,  we properly classify the possible bandwidth regimes, depending quantitatively on the growth of the scaling sequences $\lbrace S_j: j\geq 1\rbrace$. Then, we describe them through sequences of bandwidth dilation factors $\lbrace h_j:j \geq 1\rbrace$ built over sequences of shifts that exhibit regularly varying growth, characterized by a regularity exponent $p$ and other growth parameters. 
In other words, the ratio $h_j=\frac{S_{j+1}}{S_{j}}$ directly controls how much the  multipole center increases from one level to the next. We fix $S_0 = 1$, and the center sequence satisfies 
\begin{equation*}
	S_j = \prod_{k=0}^{j-1} h_k = \exp\left(\frac{\gamma(j)j^{p+1}}{p+1}\right),
\end{equation*}
where$p\in \mathbb{R}$ is the regularity index and $\gamma: \mathbb{R}^+ \to \mathbb{R}^+$ is slowly varying in the sense of Karamata (see \cite{bingham}), that is, for any $\tau>0$,
\[
\underset{j \rightarrow \infty}{\lim} \frac{\gamma(\tau j)}{\gamma(j)}=1.
\]
 %We construct the dilation factors 
%$p\in [-1,0)$ and a slowly varying function $\gamma: \mathbb{R}^+ \to\mathbb{R}^+$. 
%This framework allows us to explicitly control the scale and spacing of the dilations through the asymptotic behavior of $S_j$. 
The parameter $p$ and the function $\gamma$ governing this growth are instrumental in deriving quantitative bounds on both spectral and spatial localization of the associated functions. Furthermore, they play a central role in estimating functional norms and establishing decorrelation properties of the needlet coefficients across scales. \\

By tuning the regularity exponent $p$, we can flexibly interpolate between sparse and dense multiscale decompositions, which is crucial for balancing localization and frequency resolution.
Indeed, once we fixed $S_j$, after defining a scaling functions properly recomputed for each scale,  we construct, for any $j \geq 1$, a family of flexible-bandwidth needlets 
\begin{equation*}
	\lbrace \psi_{j,k} : k= 1 \ldots K_j\rbrace,
\end{equation*}
where $K_j$ is the total number of cubature points at level $j$ and depends explicitly on the choice of the scaling sequence $\lbrace S_j: j \geq 1 \rbrace$. 
Each index $k$ corresponds to a point $\xi_{j,k}$ on the sphere. They are typically chosen as part of a quadrature rule for $\Psi_{j}$, and each needlet $\psi_{j,k}$ is naturally associated with a spherical subregion, often called a pixel, centered at $\xi_{j,k}$.
This construction links the spectral localization of needlets to spatial regions on the sphere, allowing for localized analysis both in frequency and position.
 As in \cite{dmt24}, we will describe in what follows three distinct regimes for the sequence $\lbrace S_j:j \geq 1 \rbrace$, each determined by the asymptotic behavior of the product of the dilations $\lbrace h_j : j \geq 1 \rbrace$:
\begin{enumerate}
	\item \emph{Shrinking regime.} The scale sequence $\lbrace S_j : j \geq1 \rbrace$ grows subexponentially in 
	$j$; equivalently, the bandwidth ratio $\lbrace \Delta_j: j \geq 1 \rbrace$ shrinks to 0 while the dilation factors $\lbrace h_j: j \geq 1 \rbrace$	tends to 1 from above fast enough that 
	$\log S_j = o\left(j\right)$. Intuitively, the bands become finer and finer at a rate slower than a fixed exponential.
	\[
	\underset{j \rightarrow \infty}{\lim} \frac{S_j}{j}=0; \quad 	\underset{j \rightarrow \infty}{\lim} \Delta_j =0; \quad	\underset{j \rightarrow \infty}{\lim} h_j =0.
	\]
	\item \emph{Stable (standard) regime.} The bandwidth ratio $\lbrace \Delta_j: j \geq 1 \rbrace$ and the dilation factors $\lbrace h_j : j \geq 1 \rbrace$ are constant (e.g., 
	$h_j=B$, with $B>1$, so that $S_j = B^j$ grows exactly exponentially in $j$. This recovers the classical fixed‐bandwidth needlet construction, where each spectral window is a fixed dilation of a single prototype.
		\[
	\underset{j \rightarrow \infty}{\lim} \frac{S_j}{j}=c>0; \quad 	\underset{j \rightarrow \infty}{\lim} \Delta_j =c^\prime>0; \quad	\underset{j \rightarrow \infty}{\lim} h_j =c^{\prime \prime}>0.
	\]
	\item \emph{Spreading regime.} Here $\lbrace S_j :j  \geq 1 \rbrace$ grows superexponentially in $j$, that is, $\log S_j$ grows faster than linearly. Equivalently, $\Delta_j$ and $h_j$ themselves grow without bound, so bands spread apart at increasing scales.	
		\[
	\underset{j \rightarrow \infty}{\lim} \frac{S_j}{j}=\infty; \quad 	\underset{j \rightarrow \infty}{\lim} \Delta_j =\infty; \quad	\underset{j \rightarrow \infty}{\lim} h_j =\infty.
	\]
\end{enumerate}
We will also provide a more detailed bound for the spatial localization. Indeed, in our flexible‐bandwidth setting, fixed $M \in \mathbb{N}$, we will show that there exists a constant $c_M>0$
such that the needlet $\psi_{j,k}\left(x\right)$, for $x \in \mathbb{S}^2$, satisfies
\begin{equation*}
	\left \vert \psi_{j,k}(x) \right \vert \leq c_M\left(\frac{S^2_{j+1}}{S_{j-1}}-S_{j-1}\right)\times
	\begin{cases}
		\frac{1}{\left(S_{j-1}d_{\mathbb{S}^2} \left(x,\xi_{j,k}\right)\right)^{2M}} & \text{shrinking case}\\
		%	\frac{1}{\left(S_{j-1}\Theta_{x,y}\right)^{2M}} & \text{stable case}\\
		\frac{1}{\left(\left(S_{j}-S_{j-1}\right)d_{\mathbb{S}^2} \left(x,\xi_{j,k}\right)\right)^{2M}} & \text{spreading case}\\
	\end{cases}.
\end{equation*}
Moreover, within each of the three main regimes, we further link each case with the regularity exponent, according to the detailed behavior of the sequence $\lbrace h_j: j \geq 1\rbrace$.\\
In particular, when $\lbrace S_j : j \geq 1 \rbrace$ grows subexponentially (the shrinking regime), the resulting needlets exhibit extremely tight spatial localization on the sphere, making them especially effective for goodness‐of‐fit tests of isotropy or Gaussianity. Because the needlet coefficients at nearby cubature points become nearly uncorrelated at high $j$ one can derive precise asymptotic distributions (e.g., central limit theorems) for quadratic forms or bispectrum‐based statistics. Subexponential growth ensures that the overlap between bands is sufficiently small to make these asymptotics accurate. Also, under the hypothesis that the underlying angular power spectrum $C_\ell$ is assumed to have a certain degree of smoothness $\beta$, we will define precise criteria to choose a subexponential sequence $\lbrace S_j : j \geq 1 \rbrace$ so that each needlet band probes a narrow range of multipoles. This reveals deviations from smooth behavior—such as oscillations or localized features in $C_\ell$ more effectively than a fixed exponential construction would.
%In contrast, we will also show that 
%in the spreading regime, where $\lbrace S_j:j \geq 1\rbrace$ grows faster than any exponential, Each window around $S_j$ is very wide, covering many multipoles $\ell$-values. A wider window in $\ell$ forces the needlet to oscillate over a larger set of spherical harmonics, which produces faster decay in real space. 

\subsubsection{Plan of the paper}
Section \ref{sec:back} introduces relevant background results on harmonic analysis and spherical random fields. Section \ref{sec:fbnneed} recalls flexible bandwidth needlet projectors, their properties and their discretized version. Section \ref{sec:class} studies the relation between the scaling sequences related to the flexible bandwith needlets and the different regimes associated to their growth. Section \ref{sec:dila} studies explicit constructions for dilation sequences in the shrinking regime, while Section \ref{sec:corr} studies the high-frequency uncorrelation properties of needlet coefficients in this regime. Finally, Section \ref{sec:proofs} collects all the proofs

\section{Preliminary results}\label{sec:back}
In this section, we present the some background results necessary for the development of our main findings. We begin by reviewing essential concepts from harmonic analysis on the sphere and the theory of spherical random fields, with a focus on their spectral decomposition and the role of the angular power spectrum in characterizing statistical properties. This provides the foundation for analyzing spatially localized features in spherical data through frequency-domain techniques, see for example \cite{MaPeCUP}.  

\subsection{Harmonic analysis and random fields on the sphere}
Before introducing needlets, we briefly review some standard background material on harmonic analysis on the sphere. For further discussion and comprehensive details, we refer the reader to \cite{atki,MaPeCUP} and the references therein.
Let $x=\left(\vartheta,\varphi\right)\in\mathbb{S}^2$ be a location on the sphere; $\vartheta \in \left(0,\pi\right)$ is the colatitude,  the angle from the positive 
$z$-axis, while $\varphi\in\left[\left.0,2\pi\right)\right.$ is the longitude, the angle from the positive $x$-axis in the $xy$-plane. Let us denote the space of square-integrable functions with respect to Lebesgue measure on the sphere $dx$ by $L^{2}\left( \mathbb{S}^{2}\right) =L^{2}\left(\mathbb{S}^{2},dx\right) $. The following decomposition holds:%
\begin{equation*}
	L^{2}\left( \mathbb{S}^{2}\right) =\bigoplus_{\ell =0}^{\infty }
	\mathcal{H}_{\ell } ,
\end{equation*}%
where $\mathcal{H}_{\ell}$ is the $\left(2\ell+1\right)$-dimensional space of harmonic and homogeneous polynomials of degree $\ell $ defined on $\mathbb{R}^{3}$ and restricted to $\mathbb{S}^{2}$ (see again \cite{atki}). For any $\ell \geq 0 $, the set $\lbrace  Y_{\ell ,m}:m=-\ell,\ldots
,\ell \rbrace$ 
describes an orthonormal basis on $\mathcal{H}_{\ell}$.  Consequently, any square-integrable function $f\in L^{2}\left( \mathbb{S}%
^{2}\right) $ admits the following harmonic expansion
\begin{equation*}
	f\left( x\right) =\sum_{\ell \geq 0}\sum_{m=-\ell}^{\ell}a_{\ell
		,m}Y_{\ell ,m}\left( x\right) ,  \label{eq:exp}
\end{equation*}%
where, for $\ell \geq 0$ and $m=-\ell,\ldots ,\ell$,
\begin{equation*}
	a_{\ell ,m}=\int_{\mathbb{S}^{d}}\overline{Y}_{\ell ,m}\left( x\right)
	f\left( x\right) dx\in \mathbb{C},
\end{equation*}%
are the so-called spherical harmonic coefficients. 
Spherical harmonics can be also read as the  eigenfunctions of the spherical Laplace-Beltrami operator with eigenvalues $-\ell(\ell+1)$, such that%
\begin{equation*}
	\Delta _{\mathbb{S}^{2}}Y_{\ell,m }=-\ell (\ell +1)Y_{\ell,m }, \quad \ell
	=0,1,2,\ldots.
\end{equation*}
Also, the following \emph{addition formula} holds%
\begin{eqnarray*}
	Z_{\ell }\left( x_{1},x_{2}\right) &=&\sum_{m=-\ell}^{\ell}\overline{Y}%
	_{\ell ,m}\left( x_{1}\right) Y_{\ell ,m}\left( x_{2}\right)  \notag \\
	&=&\frac{2\ell +1}{4\pi} P_{\ell }\left( \left\langle x_{1},x_{2}\right\rangle \right) ,\quad
	\text{for }x_{1},x_{2}\in \mathbb{S}^{d},  \label{eq:harmproj}
\end{eqnarray*}%
where $\left\langle \cdot ,\cdot \right\rangle $ is the standard scalar
product over $\mathbb{R}^{3}$, $P_{\ell }$ is
the Legendre polynomial of degree $\ell $,  
\begin{equation*}
	P_{\ell }\left( \cdot\right) :[-1,1]\rightarrow \mathbb{R}, \quad P_{\ell
	}\left( t\right) :=\frac{d^{\ell }}{dt^{\ell }}(t^{2}-1)^{\ell }, \quad %
	\ell =0,1,2,...,
\end{equation*}%
cf., for example, \cite{atki}, Chapter 2.
The following \emph{reproducing kernel property} holds
\begin{equation*}
	\int_{\mathbb{S}^{2}}Z_{\ell}\left( \langle x,y\rangle \right) Z_{\ell
		^{\prime }}\left( \langle y,z\rangle \right) dy=Z_{\ell }\left( \langle
	x,z\rangle \right) \delta _{\ell ^{\prime }}^{\ell },  \label{eq:selfreprod}
\end{equation*}%
where $\delta _{\cdot }^{\cdot }$ is the Kronecker delta.
Additionally, $Z_{\ell }$ is the projector over the space $%
\mathcal{H}_{\ell}$, while  for any $%
f\in L^{2}\left( \mathbb{S}^{2}\right) $ its projection over the space $%
\mathcal{H}_{\ell }$ is given by
\begin{equation*}
	f_{\ell }\left( x\right) =Z_{\ell}\left[ f\right] \left( x\right) =\int_{%
		\mathbb{S}^{2}}Z_{\ell}\left( x,y \right) f\left( y\right)
	d y=\sum_{m=-\ell}^{\ell}a_{\ell ,m}Y_{\ell ,m}\left( x\right) .
	\label{eq:proj}
\end{equation*}
We now introduce the concept of spherical random fields, which form the foundational objects of interest in many applications involving data on the sphere. A spherical random field is a stochastic process defined on $\mathbb{S}^2$, assigning a random variable to each point on the sphere.  Any square-integrable spherical random field can be expanded into an infinite series of spherical harmonics, 
\begin{equation}\label{eqn:randomfield}
	T\left(x\right)  = \sum_{\ell \geq 0} \sum_{m=-\ell}^{\ell} a_{\ell,m} Y_{\ell,m}\left(x\right).
\end{equation}	
The random spherical harmonic coefficients $\lbrace a_{\ell,m}:\ell \geq 1, m=-\ell,\ldots \ell\rbrace$ of this expansion capture the contribution of each frequency multipole moment to the field. We assume that the random field is centered and isotropic, that is, invariant in distribution with respect to rotations of the coordinate system. Within this decomposition, the statistical properties of the fields such as isotropy are described by the behavior of its angular power spectrum, given by 
\begin{equation*}
	\mathbb{E}\left[a_{\ell,m} \bar{a}_{\ell^{\prime},m^{\prime}}\right] = C_{\ell} \delta_{\ell}^{\ell^\prime}\delta_{m}^{m^\prime},
\end{equation*}
which encodes the variance of the harmonic coefficients at each multipole $\ell$. Harmonic analysis thus provides a natural and powerful framework for understanding the structure and regularity of spherical random fields. 
In what follows we assume that the angular power spectrum satisfies the following regularity condition.
\begin{condition} [Regularity condition on the angular power spectrum]
\label{cond:Cl}
The angular power spectrum takes the form
\begin{equation}\label{eqn:regularity}
	C_{\ell} = G \left(\ell\right) \ell^{-\alpha}, 
\end{equation}
where $\alpha \geq 2$ is the spectral index, while the modulation $G\in C^{\infty}\left(\mathbb{R}_{+}\right)$ is a positive, smooth function such that
\begin{equation*}
	G_{\min}\leq G(u)<G_{\max}\quad \text{ for some }G_{\max}\geq G_{\min}>0,
\end{equation*} and  its $r$-th derivative satisfies 
\begin{equation}\label{eqn:regularity2}
	g^{(r)}(u) = O\left(u^{-\left(1-\beta\right)r} \right) \text{ as }u \rightarrow \infty
\end{equation}
\end{condition}
As usual, he spectral index $\alpha>2$ governs the asymptotic decay rate, ensuring sufficient regularity of the associated field. The function $G\in C^{\infty}\left(\mathbb{R}^+\right)$ introduces a smooth modulation of the power law, allowing for localized fluctuations or corrections. Boundedness guarantees that the overall profile of the spectrum does not exhibit unphysical spikes or dips, while imposing the condition \eqref{eqn:regularity2} reflects a Gevrey-type regularity, ensuring that while $G$ is infinitely differentiable, it does not oscillate too wildly at high frequencies—thereby maintaining control over the fine-scale structure of the spectrum. The parameter $\beta$ acts as a \textit{spectral flatness index}: smaller values of $\beta$ correspond to faster decay of derivatives, i.e., smoother modulation. For $\beta = 0$, we are in a strong smoothness regime, whereas as $\beta \rightarrow 1$ the modulation $G$ can display more variability, especially in higher derivatives.
An example covered by this structure and suggested in \cite{dmt24} considers a modulation as
\begin{equation*}
	G(\ell) = \sum_{p=1}^{P} c_p \left(d_p+ \sin\left(\frac{\ell^{\beta_p}}{M_p}\right)\right),
	\end{equation*}
	where $d_p>1$, $c_p,M_p>0$, $\beta_p \in (0,1)$  for $p=1,\ldots,P$.\\
This form provides a flexible yet analytically tractable framework for modeling a broad class of phenomena where both scale-dependent decay and smooth modulation are physically or statistically justified.

\section{The construction of flexible bandwidth needlets}\label{sec:fbnneed}
In this section we introduce the framework of flexible bandwidth needlets, a class of spherical wavelets constructed to offer a tunable balance between spatial localization and frequency resolution. This construction extends classical needlet systems by allowing the bandwidth parameter to vary with the frequency level, enabling greater adaptability in modeling and analyzing data with nonuniform spectral characteristics. These tools will be central to our subsequent investigation of high-frequency uncorrelation and asymptotic behavior of needlet coefficients. More details on the construction of this needlet system can be found in \cite{dmt24}.
\subsection{Flexible needlet projectors and weight functions}
The flexible bandwidth needlet kernel, as introduced by \cite{dmt24}, can be
defined as a weighted sum of harmonic projector: for any $j=1,2,...$%
\begin{equation*}
	\Psi _{j}\left( x,y\right) %=\sum_{\ell \geq 0}b_{j}\left( \ell \right)Z_{\ell ,d}\left( \left\langle x,y\right\rangle \right) 
	=\sum_{\ell \geq
		0}b^2_j\left( \ell\right) Z_{\ell }\left( \left\langle
	x,y\right\rangle \right) \text{ ,}
\end{equation*}%
where % for any$j \geq 1$, and% $B>1$ is a fixed (bandwidth) parameter, and 
$\lbrace b_j(\cdot):j\geq 1 \rbrace$ is sequence of a weight functions $ b_j(\cdot):\mathbb{R}%
\rightarrow \mathbb{R}$. 
According to \cite{dmt24},  we introduce the \emph{window function support} by means of the
\textit{scales or centers of the support windows.} Let $\left\{ S_{j}: j \geq 1\right\}$ denote a strictly increasing sequence of positive real numbers, defining the \emph{scales}, or \emph{centers} of the support window such that for any $j \geq 1$ the weight function $b_j\left(\cdot\right)$ is non-zero only on $\left[S_{j-1},S_{j+1}\right]$. \\
For each $j \geq 1$, each weight function must then satisfy three
properties:
\begin{enumerate} 
	\item \emph{Compact support in the frequency domain.} The weight function $b_j$ is compactly supported in $[S_{j-1},S_{j+1}];$ Also, $b_{j}\left(S_{j-1}\right) = b_{j}\left(S_{j+1}\right) = 0$, while $b_j\left(S_j\right)=1$. 

 \item \emph{Derivability.} Each $b_j$ is $C^{\infty }\left(\left[S_{j-1},S_{j+1}\right]\right)$; also it holds that 
\begin{equation*}
	\left \vert b_j^{(n)} (u) \right \vert \leq \frac{C_n}{\left(S_j-S_{j-1}\right)^n},
\end{equation*}
where $c_n>0$ and $b_j^{(n)}$ is the $n$-th derivative of $b_j$
 
\item  \emph{Partition of Unity property.} For all $\ell \geq 1$, it holds that
\[\sum_{j}b_j^{2}(\ell)=1.\]
\end{enumerate}
Each function $b_j$ is strictly localized in frequency around the central value $S_j$, and its support is limited to the interval (named window support)$\left[S_{j-1},S_{j+1}\right]$, ensuring that $b_j$ is null outside this range. This is key for maintaining semi-orthogonality of the system, typical of standard needlets: specifically, the supports of two elements of the weight function sequences $b_{j}(\cdot)$ and $b_{j^{\prime }}(\cdot)$ are disjoint whenever $\left\vert j-j^{\prime }\right\vert \geq 2$. \\
The normalization condition $b_j\left(S_j\right)=1$ guarantees that the central frequency in the band $j$ is fully preserved. This localization in the frequency domain enables multiscale analysis: each $j$-level captures information primarily from a narrow frequency band centered at $S_j$.\\
The infinitely differentiability of the function $b_j$ ensures that ineedlets are smooth and spatially well localized.
The bound on derivatives reflects the bandwidth width: as the window support  
$\left[ S_{j-1},S_{j+1}\right]$ becomes narrower, $b_j$ must transition more sharply from 0 to 1 and back to 0, leading to larger derivatives. Thus, the derivative bound ensures control over the smoothness of the corresponding needlet frame elements, which is crucial for approximation properties and for proving our central limit theorems.
This trade-off between frequency and spatial localization is intrinsic to wavelet constructions.\\
The partition of unity property guarantees that the needlet system covers the frequency axis completely and without redundancy, enabling stable decomposition and reconstruction of any function belonging to $L^2\left(\mathbb{S}^2\right)$.
More in detail, Under these conditions, in \cite{dmt24} the following space
localization property is established; for all $(x,y)\in \mathbb{S}^{2}$ and
for all integers $M$, there exist a constant $C_{M}$ (depending on $b_j(\cdot),$
but not on $x,y$ or $j$) such that%
\begin{equation}
	\begin{split}
	\left\vert \Psi _{j}\left( x,y\right) \right\vert &\leq C_{M}  \left(S_{j+1}^2- S_{j+1}^2\right)\\& \max\left(\frac{1}{\left( 1+S_{j-1}\Theta(x,y)\right)^{2M}} ,  \frac{1}{\left( 1+\left(S_{j}-S_{j-1}\right)\Theta(x,y)\right)^{2M}}\right), \label{eq:StanLocProp}
	\end{split}
\end{equation}
where $\Theta (x,y)=\arccos (\left\langle x,y\right\rangle )$ denotes the
standard geodesic distance on the unit sphere. 

We now present some results on the discretization of the flexible bandwidth needlets continuous transform discussed above. A general spherical cubature formula, providing excellent bounds on the cubature weights, was established in \cite{npw1}. 
\subsection{Discrete Construction of Flexible Bandwidth Needlets}
In this subsection, we provide an explicit discrete construction of flexible bandwidth needlets, extending the classical needlet framework by allowing scale-dependent frequency localization through a non-geometric sequence of bandwidths.  These results have been effectively applied in the context of spherical Voronoi cells, as shown in \cite[Proposition 1]{bkmpAoS} (see also \cite{bkmpAoSb}). In what follows, we will simply  adapt the findings of \cite{npw1,npw2} to our setting (see also \cite{gm2,dmt24}).\\
Let us consider the  \emph{center sequence} $\lbrace S_{j}: j\in \mathbb{N}_{0}\rbrace $, such that $S_{0}=1$,  and $S_{j+1}>S_{j}$. 
%Using $S_{j+1}=h_{j}S_{j}$, we have that
%\begin{equation}
%	S_{j}=\prod_{i=0}^{j-1}h_{i}S_{0}=\prod_{i=0}^{j-1}h_{i}. \label{eq:prodSj}
%\end{equation}%
%Note that $h_{j}>1$ for $j\geq 0$, while $\lim _{j \rightarrow \infty} h_j =1$. 
We fix $S_{-1}=0$, so that we can define a sequence of scaling functions $a_{j}:\mathbb{%
	R^{+}}\rightarrow \left[ 0,1\right] $ such that for $j \geq 1$
\begin{equation*}
	a_{j}\in C^{\infty }\left( \mathbb{R^{+}}\right) \text{ },\text{ }%
	a_{j}\left( u\right) =1\text{ for }\left\vert u\right\vert \leq S_{j-1},
\end{equation*}%
(so that $a_{0}\left( 0\right) =1)$, and
\begin{equation*}
	0<a_{j}\left( u\right) \leq 1\text{ for }u\in \left[ 0,S_{j}\right] \text{ }.
\end{equation*}
For any $j\geq$, the weight function $b_j$ is obtained by the formula 
\begin{equation}
	b_{j}\left( u\right) =\sqrt{a_{j+1}\left( u\right) -a_{j}\left( u\right) }.
	\label{eq:bjfun}
\end{equation}%
In what follows, we recall a construction of the weight functions $b_j$ as suggested by \cite{dmt24}, which relies on a smooth interpolation between frequency bands. Two equivalent and complementary strategies are available. The first defines a set of scaling functions $a_j(u)$ explicitly via a compactly supported $C^{\infty}$-mollifier $\phi_1$, followed by a scaled transition using its distribution function $\phi_2$. This yields a direct formula involving a piecewise definition adapted to the support interval $\left[S_{j-1},S_{j+1}\right]$. Alternatively, one may define a universal template function 
$a(u)$, with supported on $[0,2]$, and apply a linear rescaling $\tau_j(u)$ to align it with the desired interval. This strategy separates the smoothing profile from the scaling behavior, offering greater modularity and numerical efficiency. Both approaches are equivalent in the resulting $b_j$, defined by the difference \eqref{eq:bjfun}, ensuring the required partition of unity properties in the spectral domain. The numerical scheme is also largely analogous to the construction proposed in \cite{bkmpAoS} for standard needlets.
\begin{example} [Explicit construction of $b_{j}$]
	\begin{enumerate}
\item \emph{Step 1.} Define the $C^{\infty }$ mollifier compactly supported in $%
		\left[ -1,1\right] $:
		\begin{equation*}
			\phi _{1}(t)=%
			\begin{cases}
				\exp \left( -\frac{1}{1-t^{2}}\right) & \mbox{ for }t\in \lbrack -1,1] \\
				0 & \mbox{ otherwise }%
			\end{cases}%
			.
		\end{equation*}
\item \emph{Step 2.} Define the corresponding distribution function
		\begin{equation*}
			\phi _{2}\left( x\right) = 
			\begin{cases}
				0 & \mbox{ for }x\leq -1\\
				\frac{\int_{-1}^{x}\phi _{1}\left( t\right) dt}{%
					\int_{-1}^{1}\phi _{1}\left( t\right) dt} & \mbox{ for } x\in (-1,1)\mathbb{\ }\\
				=1 & \mbox{ for }x\geq 1
			\end{cases}.
		\end{equation*}%
		Observe that
		\begin{equation*}
			\int_{-1}^{1}\phi _{1}\left( t\right) dt=\int_{-1}^{1}\exp \left( -\frac{1}{%
				1-t^{2}}\right) dt\simeq 0.444\text{ .}
		\end{equation*}
		
\item \emph{Step 3.} Analogously to \cite{bkmpAoS}, we build $a_{j}$ from $\phi _{2}$ by a change of variable:
		\begin{equation*}
			a_{j}(u)=%
			\begin{cases}
				1 & \text{ for }u\in \left[ 0,S_{j-1}\right] \\
				\phi _{2}\left( \frac{\left( S_{j}+S_{j-1}-2u\right) }{\left(
					S_{j}-S_{j-1}\right) }\right) & \text{ for }u\in \left( \left. S_{j-1},S_{j}%
				\right] \right. \\
				0 & \text{ for }u\in \left[ \left. S_{j},\infty \right) \right.%
			\end{cases}	.
		\end{equation*}%
		An analogous construction is obtained by defining  
		\begin{equation*}
			a(x)=%
			\begin{cases}
				1 & \text{ for }x\in \left[ 0,1\right] \\
				1-\phi _{2}\left( x\right) & \text{ for }x\in \left( \left. 1,2\right]
				\right. \\
				0 & \text{ for }x\in \left[ \left. 2,\infty \right) \right.%
			\end{cases}	.
		\end{equation*}%
		and thento apply the linear scaling 
		\begin{equation*}
			\tau _{j}(u):=\frac{\left(
				2u-S_{j}-S_{j-1}\right) }{\left( S_{j}-S_{j-1}\right) },
		\end{equation*}%
		such that
				\begin{equation*}
			a_{j}(u)=a(\tau _{j}(u)).
		\end{equation*}	
		%		Note also that%
		%		\begin{eqnarray*}
			%			a_{j}^{(r)}(u) &=&-\frac{2^{r}}{\left( S_{j}-S_{j-1}\right) ^{r}}\phi
			%			_{2}^{(r)}\left( \tau _{j}(u)\right) \\
			%			&=&-\frac{2^{r}}{\left( S_{j}-S_{j-1}\right) ^{r}}\frac{\phi
				%				_{1}^{(r-1)}\left( \tau _{j}(u)\right) }{0.444}\text{ .}
			%		\end{eqnarray*}%
		%		In particular%
		%		\begin{eqnarray*}
			%			b_{j}^{\prime }\left( u\right) &=&\frac{a_{j+1}^{\prime }\left( u\right)
				%				-a_{j}^{\prime }\left( u\right) }{2\sqrt{a_{j+1}\left( u\right) -a_{j}\left(
					%					u\right) }} \\
			%			&=&
			%		\end{eqnarray*}
		
	\item \emph{Step 4} Finally, 
		\begin{equation*}
			b_{j}^{2}\left( u\right) =a_{j+1}\left( u\right) -a_{j}\left( u\right) .
		\end{equation*}
	Figure \ref{fig:example} presents a graphical example of the scaling and weight functions at selected resolution levels.
	\end{enumerate}
	\begin{figure}[htbp]
		\centering
		\includegraphics[width=0.48\textwidth]{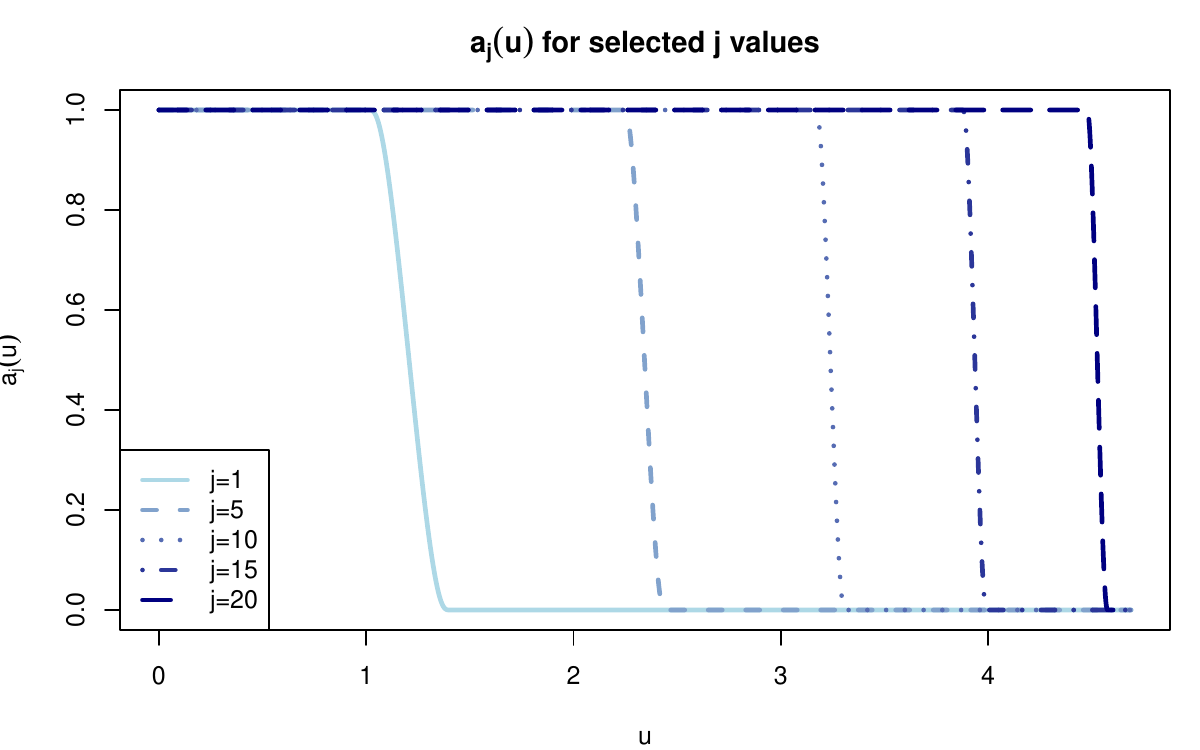}
		\hfill
		\includegraphics[width=0.48\textwidth]{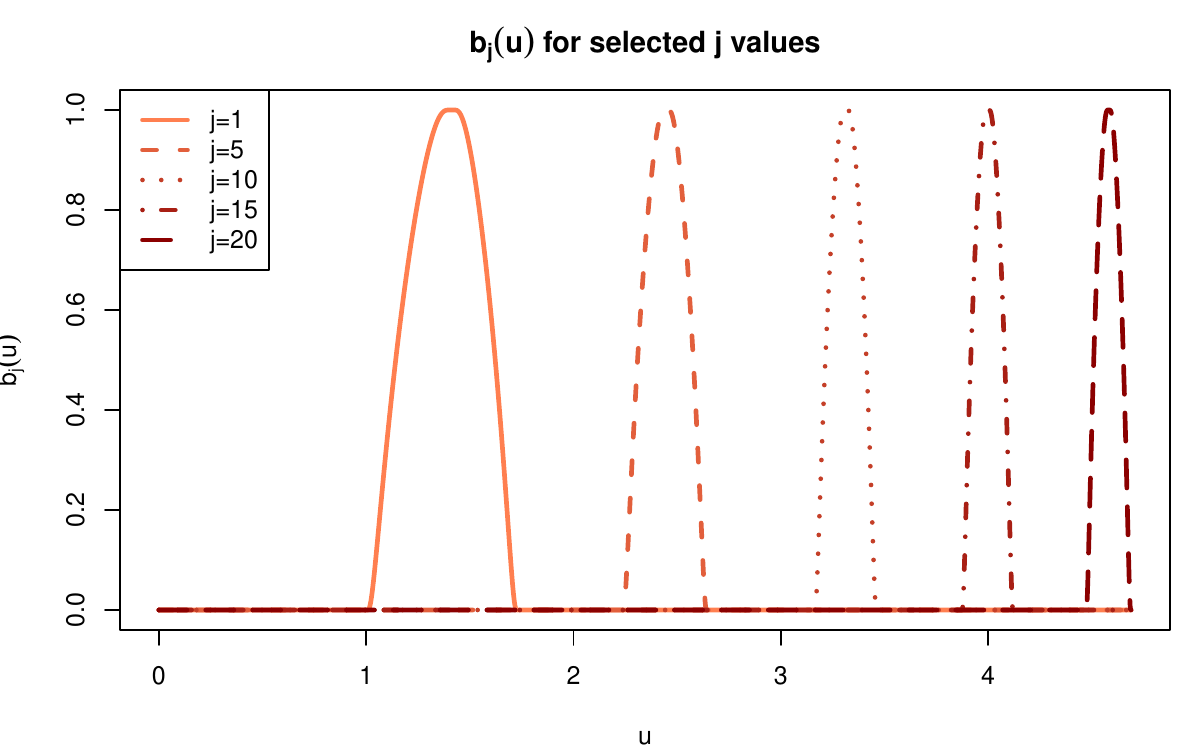}
		%\vspace{2.5cm}
		\caption{\textbf{Flexible bandwidth scaling and weight functions.} 
			The left panel displays some examples of scaling functions $a_j(u)$, which form a smooth partition of unity constructed via rescaled mollifiers. 
			Each $a_j$ transitions smoothly from $1$ to $0$ over the dyadic interval $[S_{j-1}, S_j]$, where $S_0 = 1$ and $S_{j+1} = (1 + 1/j)^\alpha S_j$ for a fixed $\alpha > 0$.
			The right panel shows the associated functions $b_j(u)$, defined by $b_j^2(u) = a_{j+1}(u) - a_j(u)$, which quantify the smooth difference between adjacent $a_j$’s and will be used to construct smooth needlet projectors.}
		\label{fig:example}
	\end{figure}
\end{example}

For any $x\in \mathbb{S}^{2}$, the flexible bandwidth needlets are defined by
\begin{eqnarray*}
	\psi _{j,k}\left( x\right) :=  %&&\sqrt{\lambda _{j,k}}M_{j}\left( x,\xi_{j,k}\right) \\
	&=&\sqrt{\lambda _{j,k}}\sum_{\ell \in \Lambda _{j}}b_{j}\left( \ell \right)
	Z_{\ell}\left( \langle x,\xi _{j,k}\rangle \right) ,\quad j\geq
	0,k=1,\ldots ,K_{j},
\end{eqnarray*}%
while, the shrinking needlet coefficients are given by
\begin{equation*}
	\beta _{j,k}=\langle f,\psi _{j,k}\rangle _{L^{2}\left( \mathbb{S}%
		^{2}\right) },\text{ for }f\in L^{2}\left( \mathbb{S}^{d}\right) ,\quad
	j\geq 1,k=1,\ldots ,K_{j},
\end{equation*}%
such that, for any $f\in L^{2}\left( \mathbb{S}^{d}\right) $,
\begin{eqnarray*}
	\beta _{j}\left( x\right) &=&\int_{\mathbb{S}^{d}}f\left( y\right) \Psi
	_{j}\left( x,y\right) dy  \notag \\
	&=&\sum_{k=1}^{K_{j}}\beta _{j,k}\psi _{j,k}\left( x\right) .
	\label{eq:needfield}
\end{eqnarray*}%
Following \cite[Eq. (43)]{npw1}, and \cite{dmt24}, the cubature points
and weights are chosen such that
%  for any $j \geq 1$, the sets of cubature points are almost uniformly $\varepsilon _{j}$-distributed with $\varepsilon _{j}:=cS_{j+1}^{-d}$ and the
%corresponding cubature weights are so that
\begin{equation}
	\lambda _{j,k}\approx S_{j-1}^{-2},\text{ \quad \quad \quad}K_{j}\approx S_{j-1}^{2},  \label{eqn:cubature}
\end{equation}%
where, for any two sequences $\left\{ a_{j},b_{j}\in \mathbb{R}\right\}
_{j\in \mathbb{N}}$ the notation $\left\{ a_{j}\approx b_{j}\right\} _{j\in
	\mathbb{N}}$ indicates that there exists a positive real valued constant $%
c\in \mathbb{R}$ such that $c^{-1}a_{j}\leq b_{j}\leq ca_{j},$ for all $j\in
\mathbb{N}$.
\subsubsection{Properties of the flexible bandwidth needlets.}
We now summarize the main properties of flexible bandwidth needlets, highlighting how their construction allows for precise control over frequency localization, spatial decay, and smoothness across multiple regimes.
\paragraph{Tightness of the frame and reconstruction formula.}
Analogously to needlets, our wavelet system is a tight frame. Frames can be
viewed as a redundant basis. More specifically, a countable family of
functions $\left\{ e_{k}:k\geq 1\right\} $ is a frame if there exist
constants $c,C\in \mathbb{R}^{+}$, named \emph{frame bounds}, such that for all $f\in
L^{2}\left( \mathbb{S}^{d}\right) $, it holds that
\begin{equation*}
	c\left\Vert f\right\Vert _{L^{2}\left( \mathbb{S}^{d}\right) }\leq
	\sum_{k\geq 1}\left\vert \langle f,e_{k}\rangle \right\vert ^{2}\leq
	C\left\Vert f\right\Vert _{L^{2}\left( \mathbb{S}^{d}\right) }.
	\label{eq:frame}
\end{equation*}%
The frame is called tight if $c=C$. In this case $c,C$ are the \emph{tightness
	constants}.
The flexible bandwidth needlet system $\left\{ \psi _{j,k}:j\geq
1;k=1,\ldots ,K_{j}\right\} $ provides a tight frame on $L^{2}\left( \mathbb{%
	S}^{2}\right) $ with frame bounds $c=C=1$:
\begin{equation*}
	\left\Vert f\right\Vert _{L^{2}\left( \mathbb{S}^{d}\right)
	}^{2}=\sum_{j\geq 0}\sum_{k=1}^{K_{j}}\left\vert \beta _{j,k}\right\vert
	^{2}.  \label{eq:tight}
\end{equation*}
As a straightforward consequence, the following needlet decomposition holds:
for $f\in L^{2}\left( \mathbb{S}^{d}\right) $,
\begin{equation*}
	f\left( x\right) =\sum_{j\geq 0}\sum_{k=1}^{K_{j}}\beta _{j,k}\psi
	_{j,k}\left( x\right) ,  \label{eq:reconstruction}
\end{equation*}%
in the $L^{2}$-sense.\\

\paragraph{Compact support in the frequency domain}
Flexible bandwidth needlets can be interpreted as finite, weighted linear combinations of spherical harmonic components (multipoles), with weights determined by the spectral window function $b_j$. Since $b_j$ is supported within the interval $\left[S_{j-1},S_{j+1}\right]$, each needlet $\psi_{j,k}$ involves only a finite range of multipole frequencies. Importantly, the number of active multipoles—that is, the bandwidth of the needlet—is directly governed by the growth behavior of the scaling sequence $\lbrace S_j: j\geq 1\rbrace$. In the shrinking regime, where $S_j$ grows slowly, the spectral window becomes narrower and more localized, while in the spreading regime, a faster growth of $S_j$leads to a broader window and a larger set of contributing multipoles. This flexibility enables the construction to adaptively balance frequency resolution and spatial localization discussed below.

\paragraph{Spatial localization}
As in \cite{dmt24}, it is straightforward to derive the following space localization property combining Equations \eqref{eq:StanLocProp} and \eqref{eqn:cubature} for $\lbrace \psi_{j,k}:j \geq 1; k= 1, \ldots,K_j\rbrace$i. For any $x\in \mathbb{S}^{2}$ and
for all integers $M$, there exist a constant $C_{M}$ such that%
\begin{equation}
	\begin{split}
	\left\vert \psi _{j,k}\left( x\right) \right\vert \leq & C_{M}  \left(\frac{S_{j+1}^2- S_{j+1}^2}{S_{j-1}}\right) \\ &\max\left(\frac{1}{\left( 1+S_{j-1}\Theta_{j,k}\right)^{2M}} ,  \frac{1}{\left( 1+\left(S_{j}-S_{j-1}\right)\Theta_{j,k}\right)^{2M}}\right), \label{eq:StanLocProp2}
	\end{split}
\end{equation}
where $\Theta_{j,k}:=\arccos (\left\langle x,\xi_{j,k}\right\rangle )$ denotes the
standard geodesic distance on the unit sphere. Equation~\eqref{eq:StanLocProp2} shows that if the point $x$ lies outside a spherical cap of radius approximately $S_{j-1}^{-1}$ centered at $\xi_{j,k}$, then the needlet $\psi_{j,k}(x)$ exhibits rapid decay, with a rate controlled either by $S_{j-1}$ or by the difference $S_{j}-S_{j-1}$. In the following, we will examine in more detail when each of these regimes dominates and how this affects localization behavior across scales.\\

\section{Scaling properties and regime classification}\label{sec:class}
In this section we investigate how the asymptotic behaviour of the center sequence $\lbrace S_j : j \geq 1\rbrace$ determines the concentration properties of the associated needlet system $\lbrace \psi_{j,k}: j \geq 1, k=1,\ldots, K_j\rbrace$. We introduce a classification of the growth regimes - subexponential, exponential, and superexponential - based on he scale growth of  $\lbrace S_j : j \geq 1\rbrace$, and establish a precise correspondence with the shrinking, stable, and spreading regimes. We also analyze how the dilation of the scales interacts with this classification, highlighting the impact of specific growth models within each regime. 
\subsection{Characterization of the scale growth regimes}\label{sec:char}
Let us introduce preliminarily $\left(S_{j+1}-S_{j-1}\right)$ as the \emph{width} of the support, the relative bandwidth ratio measures how broad the effective support $\left[S_{j-1},S_{j+1}\right]$ is, relative to its center $S_j$. It's a dimensionless descriptor of bandwidth, to help assessing how localized your needlet system is at scale. Note that this is analogous to the \textit{Scale width} $k_j=S_{j+1}/S_{j-1}$, detailed in \cite{dmt24}. 
More important to understand the growth of the center sequences and thus the asymptotic properties of flexible bandwidth needlets is the  \textit{relative bandwidth ratio}:
	\begin{equation*}
		\Delta_j = \frac{S_{j+1}-S_{j-1}}{S_j} ,
	\end{equation*}
 a key quantity for understanding flexible bandwidth needlets. Indeed, it captures how rapidly the effective frequency support changes between scales. This ratio directly influences the localization and resolution properties of the needlets, distinguishing the different regimes of shrinking, standard, and spreading bandwidths, and thus is essential to characterize their behavior and adaptivity across scales.
Then, the determination of the center sequence $\lbrace S_j: j \in \mathbb{N}\rbrace$ and the asymptotic behaviour of the relative bandwidth ratio allows us to classify the resulting needlet systems into three distinct regimes, standard, shrinking, and spreading. In particular, we have
\begin{equation*}
	\begin{split}
		\lim_{j\rightarrow \infty }\Delta_{j}= \begin{cases} 0 & \text{ (shrinking regime)}\\%
			c^{\prime}>0 & \text{ (stable regime)}\\ 
			\infty &  \text{ (spreading regime)} \end{cases}.
	\end{split}
\end{equation*}
A convenient way to understand heuristically this classification is to think to $S_j$ as the center of the window support $\left[S_{j-1},S_{j+1}\right]$. Then:
\begin{itemize}
	\item \textit{Shrinking regime.} When $\Delta_j \rightarrow 0$, $\left[S_{j-1},S_{j+1}\right]$ becomes negligible if compared to $S_j$. In practical terms, each needlet at level $j$ involves only a relatively small amount of multipoles. When $j$ grows, the width $\left[S_{j-1},S_{j+1}\right]$ shrinks around $S_j$ so that adjacent needlets hardly overlap. This is ideally the perfect scenario to gain very precise spectral localization at high $j$, paying the price of a poorer localization in the real domain.  
	\item \textit{Stable - or standard - regime.} In this case, $\left[S_{j-1},S_{j+1}\right] \approx  S_j$. In other words, each needlet support is always a constant fraction of its center frequency. This stable growth is used in standard contructions because it guarantess a uniform trade-off between spatial and spectral localization.
	\item \textit{Spreading regime.} Here $\left[S_{j-1},S_{j+1}\right] \approx  S_j$ outpaces $S_J$. This construction is very well localized in the real domain, but adjacent needlets share most of the same multipoles, allowing high correlation between their coefficients.
\end{itemize}
Now, we place $\lbrace S_j: j \geq 1 \rbrace$ in one of three following sequence spaces according to its growth rate, using the \emph{logarithmic growth index}:
\begin{equation*}
	L_j = \frac{\log S_j}{j}.
\end{equation*}
We define
\begin{enumerate}
	\item \emph{Subexponential scale:}
	\begin{equation*}\label{eqn:sub}
		\mathcal{S}_{\operatorname{sub}} = \lbrace \lbrace S_j: j \geq 1 \rbrace: \underset{j \rightarrow \infty}{ \limsup} L_j=0\rbrace.
	\end{equation*}
Equivalently, $\lbrace S_j: j \geq 1 \rbrace \in 	\mathcal{S}_{\operatorname{sub}} $ if and only if for every $\delta>0$ 
there exists $J_\delta$ such that
\begin{equation*}
	S_j<e^{\delta J} \quad \text{for all } j>J_{\delta}
\end{equation*}
In particular, $\underset{j \rightarrow \infty }{\lim } L_j = 0$.
\item \emph{Exponential scale:}
	\begin{equation*}\label{eqn:exp}
	\mathcal{S}_{\operatorname{exp}} = \lbrace \lbrace S_j: j \geq 1 \rbrace: \underset{j \rightarrow \infty}{ \lim }L_j = c>0.
\end{equation*}
Equivalently, $\lbrace S_j: j \geq 1 \rbrace \in 	\mathcal{S}_{\operatorname{exp}} $ if there exists $B>1$ and positive constants $c_1\leq c_2$ such that, for all sufficiently large $j$ 
\begin{equation*} 
	c_1  B^j\leq S_j \leq c_2 B^{j}
\end{equation*}	
In particular, $\underset{j \rightarrow \infty }{\lim } L_j = \log B$.
\item \emph{Superexponential scale:}
	\begin{equation*}\label{eqn:super}
	\mathcal{S}_{\operatorname{super}} = \lbrace \lbrace S_j: j \geq 1 \rbrace: \underset{j \rightarrow \infty}{ \lim }L_j = \infty\rbrace.
\end{equation*}
Equivalently, $\lbrace S_j: j \geq 1 \rbrace \in 	\mathcal{S}_{\operatorname{super}} $  if and only if for every $C>0$ there exists $J_C$  such that
\begin{equation*}
S_j>e^{Cj} \quad \text{for all } j > J_C
\end{equation*}
In other words, $\underset{j \rightarrow \infty }{\lim } L_j = \infty$.
\end{enumerate}
The following two theorems characterize the asymptotic behavior of the scaling sequence $\lbrace S_j : j \geq 1 \rbrace$ and its impact on the localization properties of the associated needlet functions $\lbrace \psi_{j,k}: j \geq 1; k=1,\ldots, K_j\rbrace$. \\
The first result (Theorem \ref{thm:regime}) establishes a precise correspondence between the growth regime of $\lbrace S_j : j \geq 1 \rbrace$
(subexponential, exponential, or superexponential) and the localization behavior of the needlet system $\lbrace \psi_{j,k}: j \geq 1; k=1,\ldots, K_j\rbrace$, identifying whether it lies in the shrinking, standard, or spreading regime, respectively.
This classification hinges on the asymptotic property of the relative bandwidth ratio $\lbrace \Delta_j: j \geq 1\rbrace$, which captures the scale width of the multipole window support at each resolution level. 

\begin{theorem}\label{thm:regime}
	Let the notation above prevail. Then it holds %Thus, for any $M \in \mathbb{N}$, there exists a positive constant $C_M>0$ such that 
	\begin{itemize}
		\item $\lbrace S_j : j \geq 1\rbrace \in	\mathcal{S}_{\operatorname{sub}}$ if and only if $\lbrace \psi_{j,k}: j \geq 1, k=1,\ldots K_j \rbrace$ belong to the \emph{shrinking regime};
		\item $\lbrace S_j : j \geq 1\rbrace \in	\mathcal{S}_{\operatorname{exp}}$ if and only if $\lbrace \psi_{j,k}: j \geq 1, k=1,\ldots K_j \rbrace$ belong to the \emph{standard regime};
		\item $\lbrace S_j : j \geq 1\rbrace \in	\mathcal{S}_{\operatorname{super}}$ if and only if $\lbrace \psi_{j,k}: j \geq 1, k=1,\ldots K_j \rbrace$ belong to the \emph{spreading regime}.
	\end{itemize}	
\end{theorem}
Figure \ref{fig:sequences} presents a graphical example of scaling sequences, bandwith ratios and logarithmic growth indexes in the three different regimes.
\begin{figure}[htbp]
%		\centering
%	\includegraphics[width=1\textwidth]{scaling_full.pdf}
	
	\centering
	\includegraphics[width=0.49\textwidth]{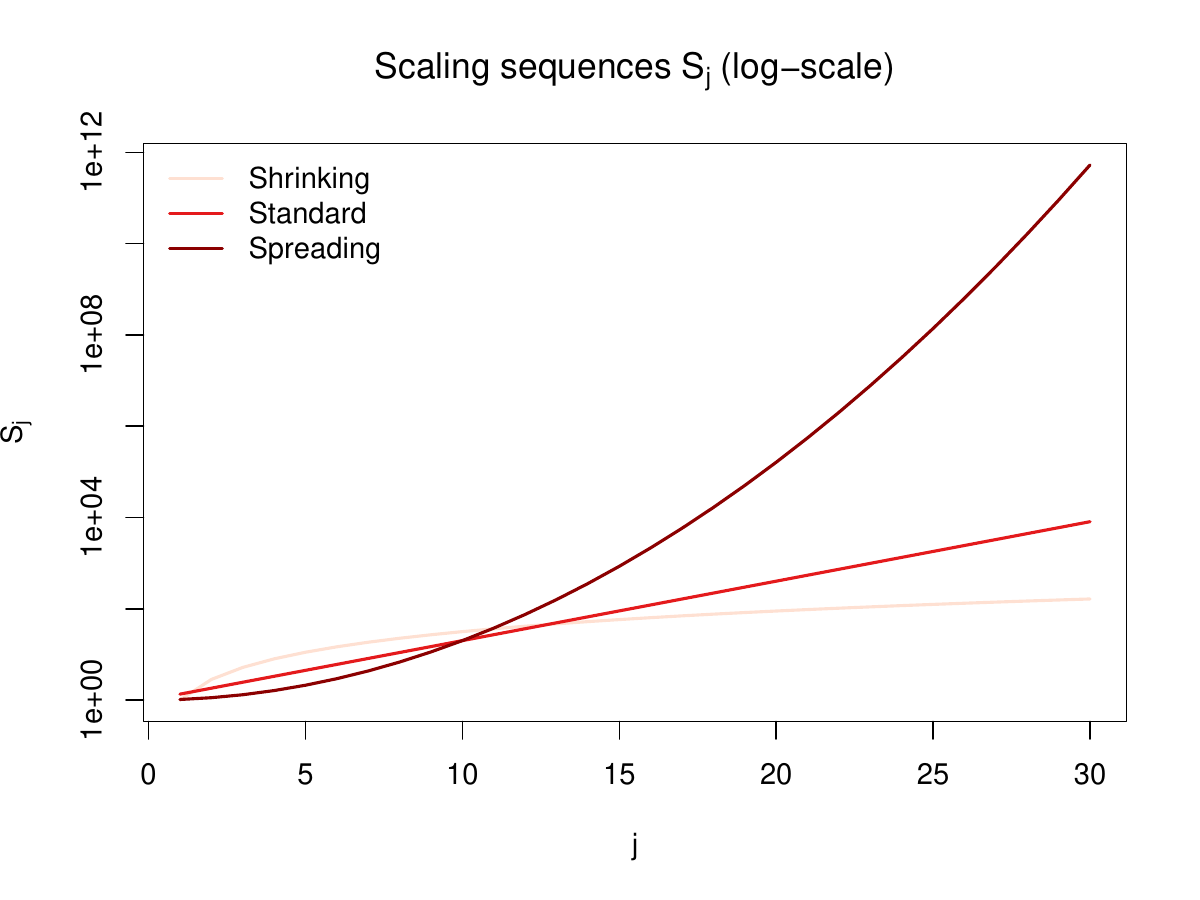}
	\hfill
	\includegraphics[width=0.49\textwidth]{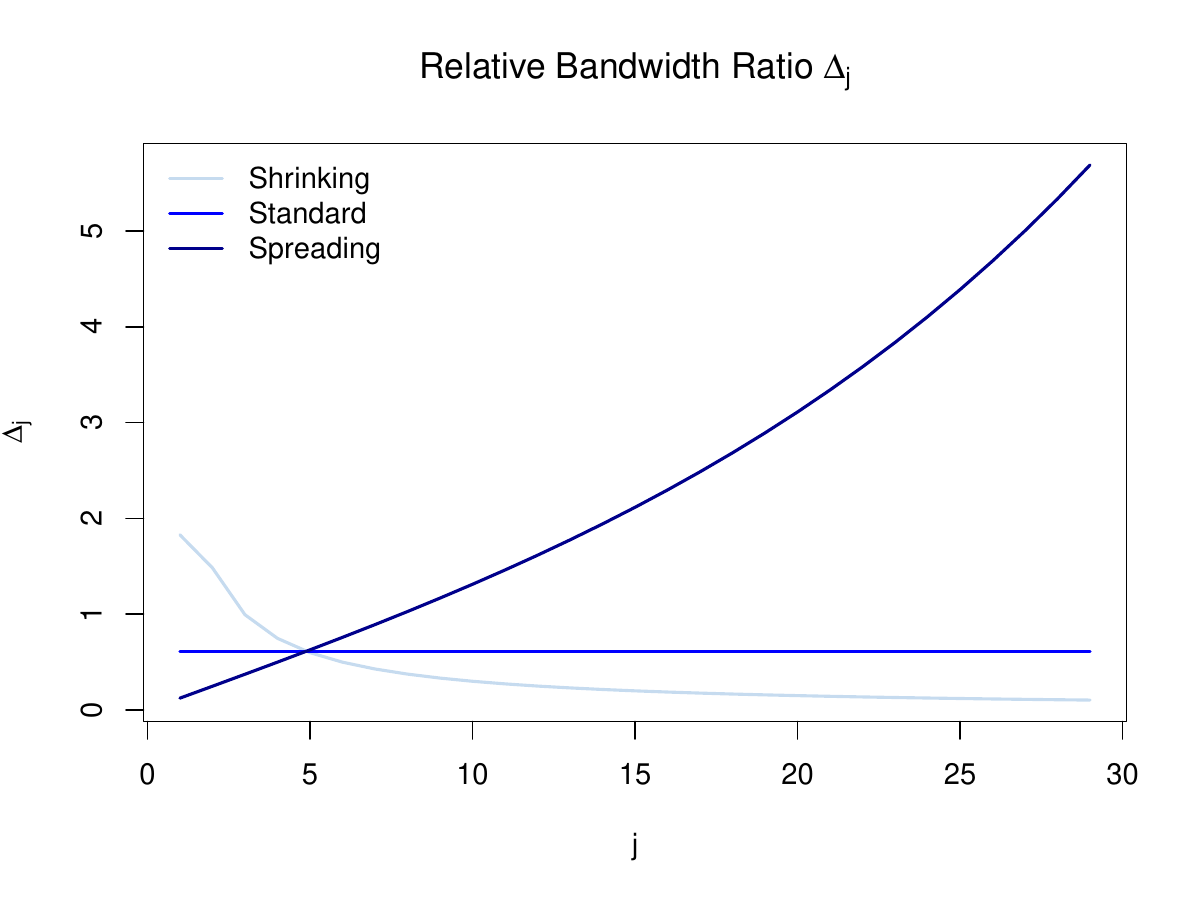}
		\hfill
	\includegraphics[width=0.49\textwidth]{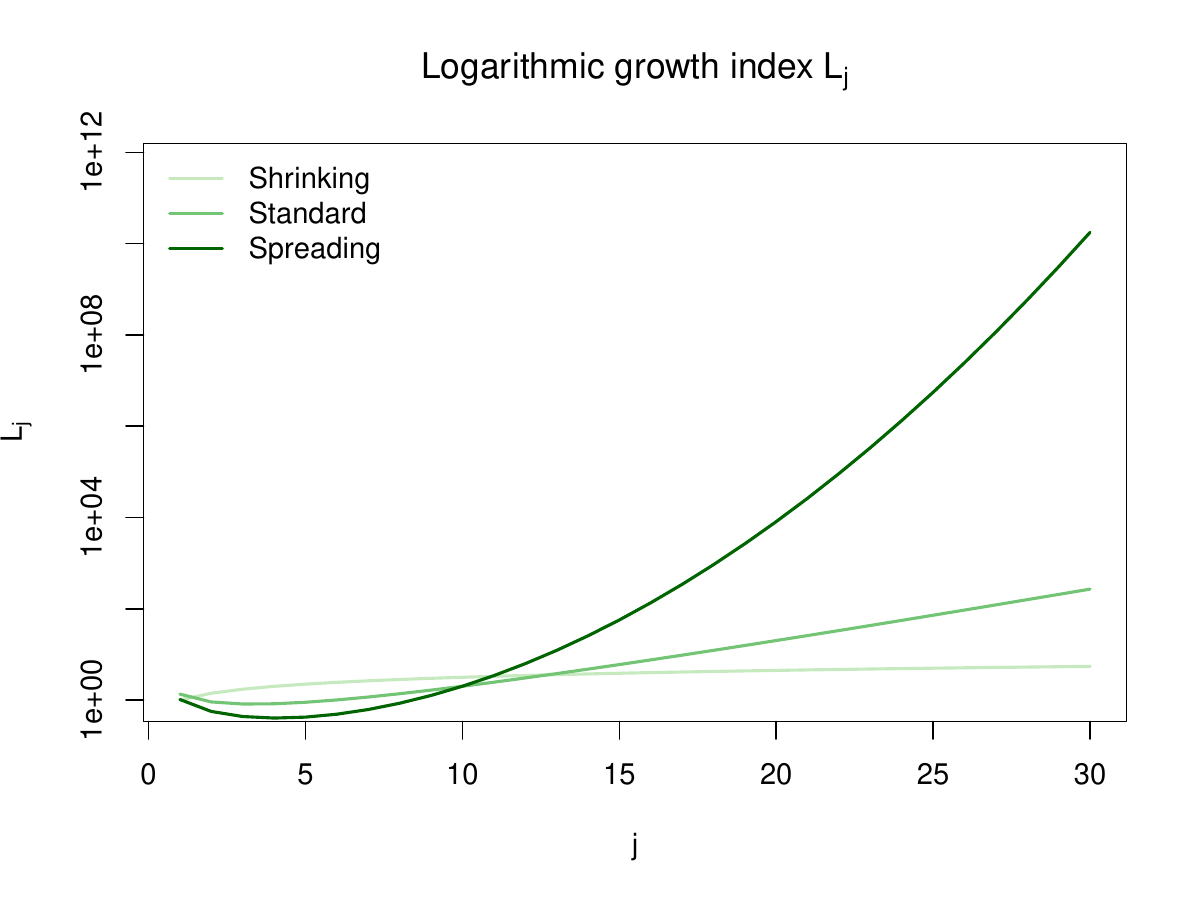}
	
	%\vspace{2.5cm}
\caption{\textbf{Scaling Sequence, Relative Bandwidth Ratio, and Logarithmic Growth Index in the Three Regimes.} 
	The top-left panel shows examples of scaling sequences $S_j$ in the shrinking, standard, and spreading regimes, plotted on a logarithmic scale. 
	The top-right panel displays the relative bandwidth ratio $\Delta_j = (S_{j+1} - S_{j-1}) / S_j$ across the three regimes. 
	The bottom panel presents the logarithmic growth index $L_j = \log S_j / j$, highlighting the different growth behaviors.}
	\label{fig:sequences}
\end{figure}
The proof of Theorem \ref{thm:regime} is available in Section \ref{sec:proofs}

 Having established the correspondence between the scaling growth behavior and the qualitative regimes of the related needlets, we now turn to a more quantitative aspect. \\
 The second result (Theorem \ref{thm:local}) provides a uniform bound on $l\left \vert \psi_{j,k} (x) \right \vert$  as $j \rightarrow \infty$, explicitly quantifying the rate of decay in terms of the scaling sequence $S_j$and thus capturing the localization strength in each regime. The proof is in Section \ref{sec:proofs}
 \begin{theorem}\label{thm:local}
 	As $j \rightarrow \infty$, for all $x \in \mathbb{S}^2$, and $M \in \mathbb{N}$, with $M >2$, there
 	exists a constant $c_M>0$ such that
 	\begin{itemize}
 		\item in the \textit{shrinking regime} ($\lim_{j \rightarrow \infty} \Delta_j =0$):  		
 			\begin{equation*}
 			\psi_{j,k} (x) \leq C_M  \left( \frac{S_{j+1}^2-S_{j-1}^2}{S_{j-1}}\right)\frac{1}{\left(1+\left(S_{j}-S_{j-1}\right)\Theta_{j,k}\right)^{M}}
 		\end{equation*}
 			\item in the \textit{standard regime} ($\lim_{j \rightarrow \infty} \Delta_j =\log B$, $B>1$):  		
 		\begin{equation*}
 			\psi_{j,k} (x) \leq C_M   \left(\frac{B^4-1}{B} \right)\frac{B^j}{\left(1+B^j\Theta_{j,k}\right)^{M}}
 		\end{equation*}
 	\end{itemize}
 \end{theorem}
  	Figure \ref{fig:decay} presents and example for the different decays for the flexible bandwith needlets at the same resolution level $j=2$ in the three regimes. Note that the shrinking regime exhibits the weakest spatial localization, although it remains quasi-exponential.
 	 \begin{figure}[htbp]
 		 	\centering
 		 	\includegraphics[width=\textwidth]{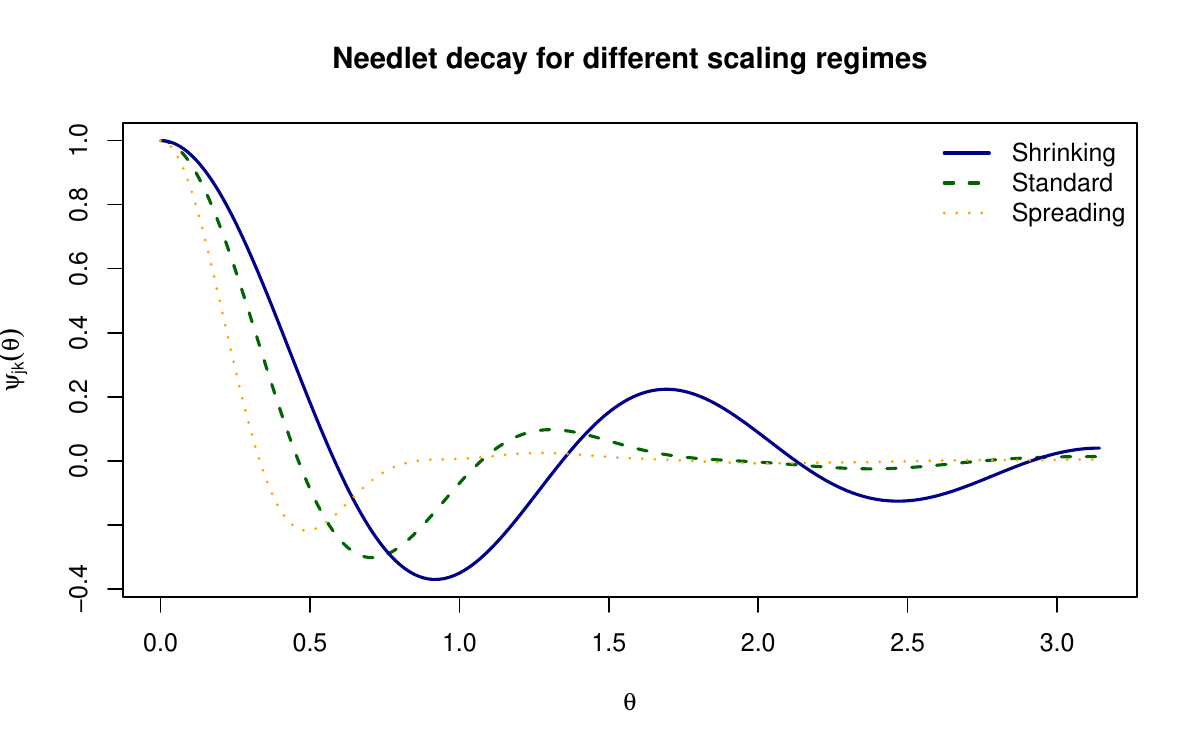}
 		 	\hfill
 		 	%\vspace{2.5cm}
 		 	\caption{\textbf{Needlet decays in the three regimes.} 
 			 		The different decay behaviors with respect to angular distances for the three regimes are qualitatively illustrated here for $j=2$ and cubature point $\xi_{jk}$ on the North Pole. }
 		 	\label{fig:decay}
 		 \end{figure}
\section{Bandwidth Dilation and Growth Regimes}\label{sec:dila}
In this section, we introduce the \emph{bandwidth dilation factor}
$$
h_j=\frac{S_{j+1}}{S_j},
$$
which measures the relative growth of $S_j$. By quantifying how rapidly $S_j$ increases, $h_j$ allows us to classify finer subregimes, each with explicit scaling laws and sharper analytical localization bounds, and will be pivotal in studying the asymptotic uncorrelation properties of the coefficients $\lbrace\beta_{j,k}:k=1,\ldots,K_j\rbrace$.
More in detail, the \textit{bandwidth dilation factor} $\lbrace h_{j}: j \geq 1\rbrace$ controls the relative enlargement of the centers across scales in the needlet system.  
%The sequence \emph{scale width} $\lbrace \kappa_j: j \geq 1 \rbrace$ of functions $b_{j}\left( \cdot \right)$, defined as the ratio between the largest and smallest value of $\ell $ in their support%, which is constant and given by $B^2$. 
	%The scale width can be read as the ratio between the larger and smaller angular scales which are probed by the needlet projector at resolution level $j$. 
%The introduction of the center sequence $\lbrace S_j: j \in \mathbb{N}\rbrace$ allows us to classify the resulting needlet systems into three distinct regimes, standard, shrinking, and spreading, according to the asymptotic behavior of the bandwidth ratio $h_j$ and the relative bandwidth ratio $\Delta_j$. 
As for $\Delta_j$, depending on its limiting behavior, the sequence ${h_j}_{j \geq 1}$ characterizes the distinct scaling regimes that significantly influence the structural properties of the needlet system.
%\begin{equation*}
%	\begin{split}
%		\lim_{j\rightarrow \infty }h_{j}= \begin{cases} 1 & \text{ (shrinking case)}\\%
%c>1 & \text{ (standard case)}\\ 
% \infty &  \text{ (spreading case)} \end{cases},
%\end{split}
%\end{equation*}
Indeed, it is the possible also to rewrite the scaling sequence as
\begin{equation*}
	S_{j}=\prod_{i=0}^{j-1}h_{i}S_{0}=\prod_{i=0}^{j-1}h_{i}, \label{eq:prodSj}
\end{equation*}%
since $S_0 =1$. Also,
\begin{equation*}
	\Delta_{j}=h_j-\frac{1}{h_{j-1}}
\end{equation*}%
%In terms of the bandwidth dilation, the concentration property for the needlet kernel can be rewritten as 
%\begin{equation}
%	\left\vert \psi _{j,k}\left( x\right) \right\vert \leq C_{M} \frac{1}{\prod_{k=0}^{j-2}h_{k}^2} \left(\right)\frac{ \prod_{k=0}^{j}h_k^2- \prod_{k=0}^{j-2}h^2_k }{\prod_{k=0}^{j-2}h_k} \left(h_{j}^2- h_{j-1}^{-2}\right) \max\left(\frac{1}{\left( 1+\frac{S_{j-1}}{h_{j-1}}\Theta\right)^{2M}} ,  \frac{1}{\left( 1+S_{j}\left(1-\frac{1}{h_{j-1}}\right)\Theta\right)^{2M}}\right), \label{eq:LocNew}
%\end{equation}
%%%\paragraph{Standard needlet case.} The key difference between standard and flexible bandwidth needlets lies in the choice of frequency localization scales, which directly affects their localization properties, asymptotic behavior, and computational adaptability (see \cite{npw1,npw2}). Use a fixed geometric progression for the frequency localization, typically with scales of the form $B^{j}$
%%B for some fixed scale $B>1$. The window width belongs to the interval $\left[ B^{j-1},B^{j+1}\right]$, while the bandwidth ratio $h_j=B$ is constant across all levels, as well as the geometric width of support band $B^2$. The concentration bound takes the form
%\begin{equation*}
%	\left\vert \Psi^{\text{standard}}_{j}\left( x,y\right) \right\vert \leq \frac{C_{M}B^{2j}}{%
%		\left\vert 1+B^{j}\Theta (x,y)\right\vert ^{M}}.
%\end{equation*}
In multiscale constructions such as flexible‐bandwidth needlets, the sequence of dilation factors
$h_j$ serves as the fundamental driver of how scales evolve and overlap.
% Depending on whether $h_j$ tends to unity, stabilizes at a constant greater than one, or grows without bound, one observes three markedly different regimes of scale behavior:\\
A convenient way to provide a classification is to define the center sequence $\lbrace S_j: j \geq 1 \rbrace$  in terms of \emph{shifts} (or \emph{perturbations}) $\lbrace \varepsilon_j : j \geq 1 \rbrace$. At each $j$, each scaling $S_j$ is built from the $j-1$ one by applying a small perturbation, such that
\begin{equation*}
h_j =  1 + \varepsilon_j. % \quad \lim_{j \rightarrow \infty}\varepsilon_j=0.
\end{equation*} 
Then
\begin{equation*}
	\begin{split}
	S_j & = S_0 \prod_{k=0}^{j} h_k \\
	& =  \exp\left(\sum_{k=0}^{j} \log \left(1+\varepsilon_j\right)\right)\\
	& \approx \exp \left(\sum_{k=0}^{j} \varepsilon_j\right).
	\end{split} \label{eqn:Sj}
\end{equation*}
In the following, we classify these regimes in detail, characterize the corresponding perturbations $\varepsilon_j$ and illustrate how each regime influences the growth of the central scales
$S_j=\exp\left(\sum_{k=0}^{j-1} \varepsilon_k\right)$.\\
This classification provides a unified, explicitly computable framework for understanding the trade-offs between localization, redundancy, and computational scalability in flexible bandwidth needlet frames. 
 It is thus natural to use sequences of regular variation to explicitly balance these competing demands, providing a unified class of dliation schemes, which guarantees explicit control on the $\lbrace S_j:j \geq 1\rbrace$ and transparent criteria for the localization trade-off. We recall here the definition of \emph{slowly} and \emph{regularly varying functions} due to J.Karamata (see for example \cite{bingham,boj17}).
\begin{definition}\label{def:karamata}
	A measurable function $\gamma:\mathbb{R}^+ \mapsto\mathbb{R}^+$ is \emph{slowly varying} if for all $a\geq0$
	\begin{equation*}
		\underset{u \rightarrow \infty}{\lim} \frac{\gamma(au)}{\gamma(u)}=1;
	\end{equation*}
	A measurable function $\gamma^\prime:\mathbb{R}^+ \mapsto\mathbb{R}^+$ is \emph{regularly varying} if for all $a\geq0$
	\begin{equation*}
		\underset{u \rightarrow \infty}{\lim} \frac{\gamma^\prime(au)}{\gamma(u)}=c \in \mathbb{R}^+.
	\end{equation*}
\end{definition}
Observe that (cf., e.g., \cite{bingham}) 
\begin{itemize}
	\item 	If $\gamma$ is slowly varying, then $\gamma^\prime(x)=\gamma(x) x^{\delta}$, with $\delta>0$ is regularly varying of index $\beta$;
	\item conversely, if $\gamma^\prime$ is regularly varying, then it necessarily has a representation of the form $\gamma^\prime(x)=\gamma(x) x^{\delta}$, with $\delta>0$ is regularly varying of index $\delta$.
	\item (Generalized Karamata Theorem) A positive sequence $\lbrace \gamma(k) : i \geq K$ is regularly varying of index $\delta$ if and only if for some $\sigma>-1-\delta$ 
	\begin{equation*}
		\underset{j \rightarrow \infty}{\lim} \frac{1}{n^{1+\sigma}\gamma(k)}\sum_{k=1}^{j}k^{\sigma}\gamma(k)=\frac{1}{\sigma+\delta+1}
	\end{equation*}
\end{itemize}
The regular variation hypothesis is thus the most natural and minimally restrictive condition under which asymptotic behaviour of the scales $\lbrace S_j: j \geq 1 \rbrace$ can be fully classified and optimally tuned.  
\begin{condition}\label{cond_epsilon}
	The shift sequence $\lbrace \epsilon _ j : j \geq 1\rbrace$ is regularly varying of order $p$, that is, ther exist a slowly regular function $\gamma:\mathbb{R}^+ \mapsto \mathbb{R}^+$ and $p \in \mathbb{R}$ such that for all $j$
	\begin{equation}\label{eqn:rv}
		\varepsilon_j = \gamma(j) j^{p}.
	\end{equation}
\end{condition}

\subsection{Dilations and asymptotic behavior of the scaling centers}
In this section, we focus on the shrinking regime, where the effective bandwidth decreases via a sequence $h_j \rightarrow 1$, as the resolution index $j$ increases. In other words, the scales freeze in the sense that $\epsilon_j$ decreases to zero, causing the bandwidth ratios to narrow at successive levels. Consequently, the width $\left[S_{j+1}-S_{j-1}\right]$ grows more slowly than the center $S_j$. However, by the uncertainty principle, there is a trade‐off: as the frequency band widens more slowly, the spatial decay of the corresponding needlet is less pronounced. In concrete terms, one obtains finer spectral resolution at each level but somewhat more spread‐out spatial support.\\
To clarify the relationship between the incremental scaling shifts $\lbrace \varepsilon_j :j\geq \rbrace$ and the scaling centers $\lbrace S_j : j \geq 1 \rbrace$, we examine how small perturbations in dilation translate into changes in ${\Delta_j}$ under the shrinking regime. Note that, since by definition
	\begin{equation*}
		S_{j}=\exp\left(\sum_{k=0}^{j-1} \varepsilon_k\right);\quad 	S_{j+1}=\exp\left(\sum_{k=0}^{j} \varepsilon_k\right);\quad 	S_{j-1}=\exp\left(\sum_{k=0}^{j-2} \varepsilon_k\right),
	\end{equation*}	
	it holds that
	\begin{equation*}
		\Delta_j = \exp\left(\varepsilon_{j}\right)-\exp\left(-\varepsilon_{j-1}\right).
	\end{equation*}	
	Since in the shrinking regime $\underset{j \rightarrow \infty}{\lim} \varepsilon_j = 0$, we expand both the exponential functions at the first order to obtain:
	\begin{equation*}
		\Delta_j \simeq \left(\epsilon_{j}+\epsilon_{j-1}\right)+\frac{1}{2} \left(\varepsilon_{j}^2-\varepsilon_{j-1}^2\right) + O\left(\varepsilon_{j}^3+\varepsilon_{j-1}^3\right)
	\end{equation*} 	
	As $j$ goes to $\infty$, $\varepsilon_{j}$ and $\varepsilon_{j-1}$ are equivalent and thus 
		\begin{equation*}
		\Delta_j \simeq 2\epsilon_{j}.
	\end{equation*} 		 
Below, we describe how the long‐term behavior of the scales $\lbrace S_j:j\geq 1\rbrace$ depends on whether the series $\sum_{j\geq0} \varepsilon_j$
converges to a finite value or diverges. These two cases yield qualitatively different bandwidth dynamics. When the series diverges, we also identify several relevant subregimes within the shrinking framework. This distinction is critical in applications such as nonparametric goodness‐of‐fit testing, where one must carefully balance spatial localization against frequency resolution. 
\begin{theorem}\label{thm:regimes}
	Under the assumptions \eqref{eqn:rv} in Condition \ref{cond_epsilon}, it holds that, as $j$ increases,
	\begin{equation*}\label{eq:Sj}
	S_j = \begin{cases}
			\exp\left(\gamma(j) \frac{j^{p+1}}{p+1}\right) & \text{ for }p\neq -1,\\
			\exp\left(\gamma(j) \log j \right)& \text{ for }p = -1, \gamma(j) \neq O\left(\left(\log j\right)^{-1}\right),\\
				\left( \log j \right)^{\eta}& \text{ for }p = -1, \gamma(j) =\eta \left(\log j\right)^{-1}, \eta>0.
		\end{cases}
	\end{equation*}
	The determination of the regime depends on the value of the regularity index and the asymptotic behavior of the slowly varying function $\gamma$ as follows:
	\begin{itemize}
		\item \emph{Shrinking regime.} In this case, it holds that 
		\[\lbrace p<0,  \text{any } \gamma\rbrace \quad \mbox{ or } \quad \lbrace p=0, \gamma:\underset{j \rightarrow\infty}{\lim} \gamma(j) = 0 \rbrace;\]
				\item \emph{Stable regime.} In this case, it holds that 
				\[\lbrace p=0, \gamma:\underset{j \rightarrow\infty}{\lim} \gamma(j) = c\rbrace,\] 
				with $0<c<\infty$;
			\item \emph{Spreading regime.} In this case, it holds that 
			\[\lbrace p>0, \text{ any }\gamma\rbrace \quad \mbox{ or }\quad \lbrace  p=0,\gamma:\underset{j \rightarrow\infty}{\lim} \gamma(j) = \infty\rbrace.\]
	\end{itemize}
	Additionally, in the shrinking regime, we can define two main subcases: 
	\begin{itemize}
		\item \emph{Totally convergent scales.} Here, it holds that \[\lbrace p<-1,\text{ any }\gamma  \rbrace \quad \mbox{ or } \quad  \lbrace p=-1, \gamma(j)=O\left(\left( \log j\right)^{-1-\delta}, \delta>0\right)\rbrace. \] In this case, $\underset{j \rightarrow \infty}{\lim}\sum_{k=0}^{j-1} \varepsilon_k = E < \infty$, and  
		\begin{equation*}
			S_j =  \exp \left( \sum_{k=0}^{j-1} \varepsilon_k\right) \underset{j \rightarrow \infty}{\longrightarrow}  S_{\infty} = e^{E} <\infty.
		\end{equation*}
		\item \emph{Divergent scales.}  Here, 
		\[\begin{split}
			&\lbrace p \in (-1,0),\quad  \text{ any }\gamma\rbrace \mbox{ or } \\ 
			&\lbrace p  = - 1, \quad \gamma(j) \geq  O \left(\left( \log j\right)^{-1+\delta}\right), \delta\geq0 \rbrace\mbox{ or } \\ 
			&\lbrace p  = 0, \quad \gamma(j):\underset{j \rightarrow \infty}{\lim}\gamma(j)=0\rbrace.\end{split}\] 
		Here, $\underset{j \rightarrow \infty}{\lim}\sum_{k=0}^{j-1} \varepsilon_k=\infty$, then the scales $\{S_j\}$ grow unboundedly but remain subexponential. In the first case we have \begin{equation*}
			S_j =  \exp \left( \sum_{k=0}^{j-1} \varepsilon_k\right) \underset{j \rightarrow \infty}{\longrightarrow}   \exp\left( \frac{j^{p+1}\gamma(j)}{p+1}\right).
		\end{equation*} 
		In the second case, if $\gamma(j)>O\left(\left(\log j\right)^{-1}\right)$, 
		\begin{equation*}
			S_j =  \exp \left( \sum_{k=0}^{j-1} \varepsilon_k\right) \underset{j \rightarrow \infty}{\longrightarrow}  j ^{\gamma(j)},
		\end{equation*}	
		while in the boundary case, $\gamma= \eta \left(\log j\right)^{-1}$, $\eta >0$
		\begin{equation*}
		S_j =  \exp \left( \sum_{k=0}^{j-1} \varepsilon_k\right) \underset{j \rightarrow \infty}{\longrightarrow}  \log j ^{\eta} .
		\end{equation*}	
		In the third case, 
		\begin{equation*}
			S_j =  \exp \left( \sum_{k=0}^{j-1} \varepsilon_k\right) \underset{j \rightarrow \infty}{\longrightarrow}  \exp\left( j \gamma(j)\right) .
		\end{equation*}	
		\end{itemize}
	
\end{theorem}
The proof is again available in Section \ref{sec:proofs}.\\
Even if in both the cases the relative bandwidth ratio - proportional to the shift - vanishes asymptotically, If we consider the total convergence case, beyond a certain index $j$, the scales $\lbrace S_j : j\geq 1\rbrace$ stabilize to a finite value. It means that after some finite index $J_\infty$, the cumulative shift has essentially saturated, so all subsequent scales satisfy $S_j \approx S_{\infty}$.  This stability recovers the classical harmonic analysis setting, since dilation parameters become effectively constant at high frequencies. More of interest, the divergent dilation case allows one to control the growth of $S_j$.   In practical applications, this unbounded but controlled growth grants extra flexibility: by tuning $p$ and the and the slowly varying component $\gamma(j)$, one can finely balance spectral concentration against spatial localization. Here we present a partial list of divergent subexponential dilations sorted by increasing power of growth.
		\begin{example}\label{cor:subregimes}
			In the shrinking regime, we can identify the following subclasses where all the shifts are regularly varying sequences of index $-p$.
		\begin{itemize} 
%			\item \textit{Pure Polylogarithmic scales.} The shifts are of the form 
%			\begin{equation*}
%				\varepsilon_j \sim \frac{\gamma}{\log j} + \frac{\delta}{j\log j \log \log j } + o\left(\frac{1}{j\log j}\right)\quad \gamma >0, \delta \in \mathbb{R}
%			\end{equation*}
%			such the scaling sequence is given by
%			\begin{equation*}
%				S_j = (\log j)^{\gamma}\left(\log \log j\right)^\delta, \quad \gamma >0, \delta \in \mathbb{R},
%			\end{equation*}
%			with 
%			\begin{equation*}
%				\Delta_j = \frac{2\gamma}{j \log j}+ o\left(\frac{1}{j\log j}\right).
%			\end{equation*}
			\item \textit{Logarithmic scales.} We take $p=-1$ and $\gamma(j)=\eta (\log j)^{-1}$, with $\eta > 0$. In this case, the shifts are of the form 
			\begin{equation*}
				\varepsilon_j \sim \frac{\eta}{j\log j},
			\end{equation*}
			which gives:
			\begin{equation*}
			S_j = (\log j)^\eta,
			\end{equation*}
			with
			\begin{equation*}
					 \Delta_j = 2\left( \frac{\eta}{j \log j}\right) + o \left(\frac{1}{j\log j}\right).
			\end{equation*}
			This case includes also modification of $\varepsilon_j$ of the type $\frac{\delta}{j\log j \log \log j } $.
%			\item \textit{Log-polynomial scales.} For $\eta,\gamma>0$, we take  
%			\begin{equation*}
%				\varepsilon_j \sim \frac{\eta}{j}+\frac{\gamma}{j \log j} + o(j \log j),
%			\end{equation*}	
%			yielding:
%			\begin{equation*}
%				S_j = j^{\eta} (\log j)^{\eta}, 
%			\end{equation*}
%			with
%			\begin{equation*}
%				\Delta_j = 2\left( \frac{\eta}{j}+\frac{\gamma}{j \log j}\right) + o\left(j \log j \right).
%			\end{equation*}
			\item \textit{Polynomial scales.} We take $p=-1$ and $\gamma(j)=\eta$, for $\eta>0$. In this case,  we have that 
			\begin{equation*}
				\varepsilon_j \sim \frac{\eta}{j},% + o(\frac{1}{j}),
			\end{equation*}	
			yielding:
			\begin{equation*}
			S_j =  j^{\eta}, 
			\end{equation*}
			with
			\begin{equation*}
					 \Delta_j = \frac{2\eta}{j }+ O\left(j^{-3}\right). 
			\end{equation*}
			This includes also log-polynomial scales, where for $\eta_1,\eta_2>0$, we have that 
				$\varepsilon_j \sim \frac{\eta_1}{j}+\frac{\eta_2}{j \log j} + o(j \log j).$
				\item \textit{Log-Power Exponential scales.} We take $p=-1$ and $\gamma(j)=\eta q \left(\log j\right)^{q-1}$, with $\eta>0$ and $q \in (0,1)$. In this case,  
			\begin{equation*}
				\varepsilon_j \sim \eta q \frac{\left(\log j\right)^{q-1}}{j}, % + o(\frac{\left(\log j\right)^p}{j}),
			\end{equation*}	
			leading to:
			\begin{equation*}
				S_j =  \exp\left(\eta \left(\log j\right)^{q}\right), 
			\end{equation*}
			with
			\begin{equation*}
				\Delta_j = 2\eta q \frac{\left(\log j\right)^{q-1}}{j} + o(\frac{\left(\log j\right)^p}{j}).
			\end{equation*}
	\item \textit{Stretched (or mild) exponential scales.} For $0 < p < 1$ and $\gamma(j)= \eta$, $\eta>0$, we have
	\begin{equation*}
		\varepsilon_j = \eta j^{-p} + o(j^{-p}),
		\end{equation*} 
		so that 
		\begin{equation*}
			S_j = \exp \left( \frac{\eta j^{1-p}}{1-p} \right),
			\end{equation*}
			where
			%()which grows faster than any polynomial but slower than any full exponential. Then
			\begin{equation*}
					 \Delta_j = 2\eta(1-p)j^{-p} + o\left(j^{-p}\right), \\
			\end{equation*}
					\end{itemize}
\end{example}
	All the cases in Example \ref{cor:subregimes} are summarized in Table \ref{tab:shrink_regimes}, while Figure \ref{fig:combined} offers a graphical illustration of the behaviour of shifts and scaling centers.
	\begin{figure}[htbp]
	\begin{subfigure}[b]{0.48\textwidth}
		\centering
		\includegraphics[width=\textwidth]{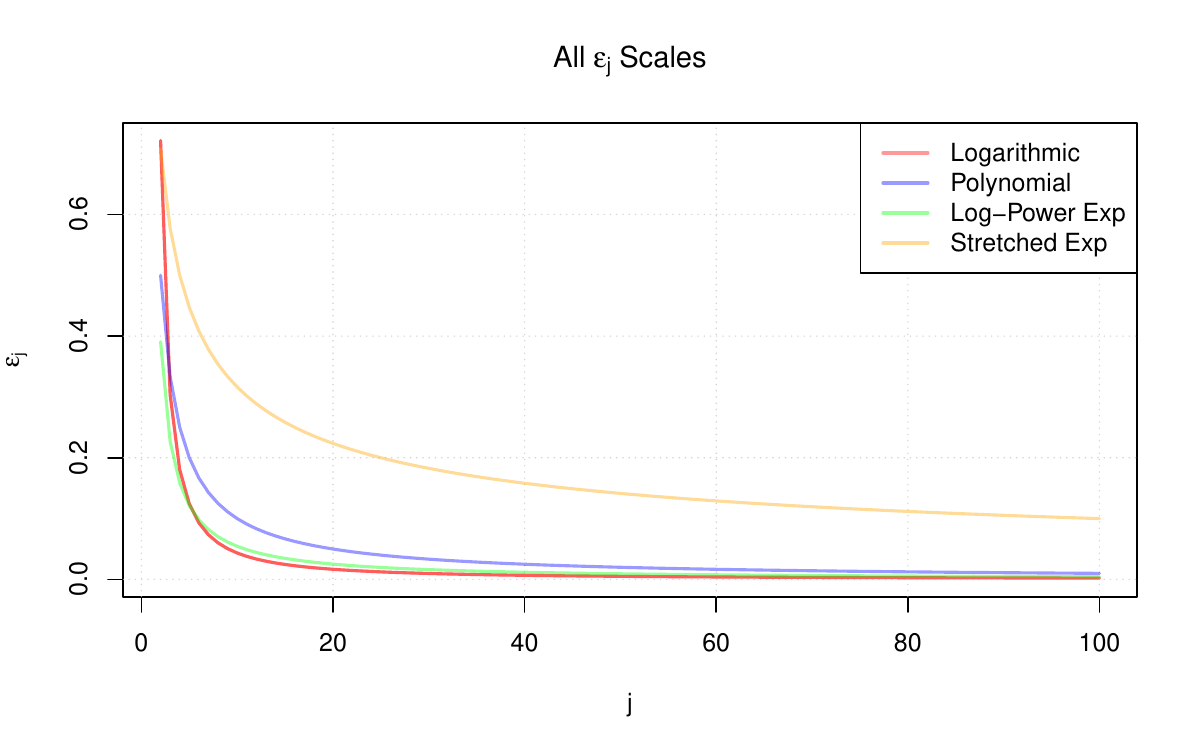}
		\caption{Decay of $\varepsilon_j$ in various regimes.}
		\label{fig:epsilon}
	\end{subfigure}
	\hfill
	% Subfigure for S_j
	\begin{subfigure}[b]{0.48\textwidth}
		\centering
		\includegraphics[width=\textwidth]{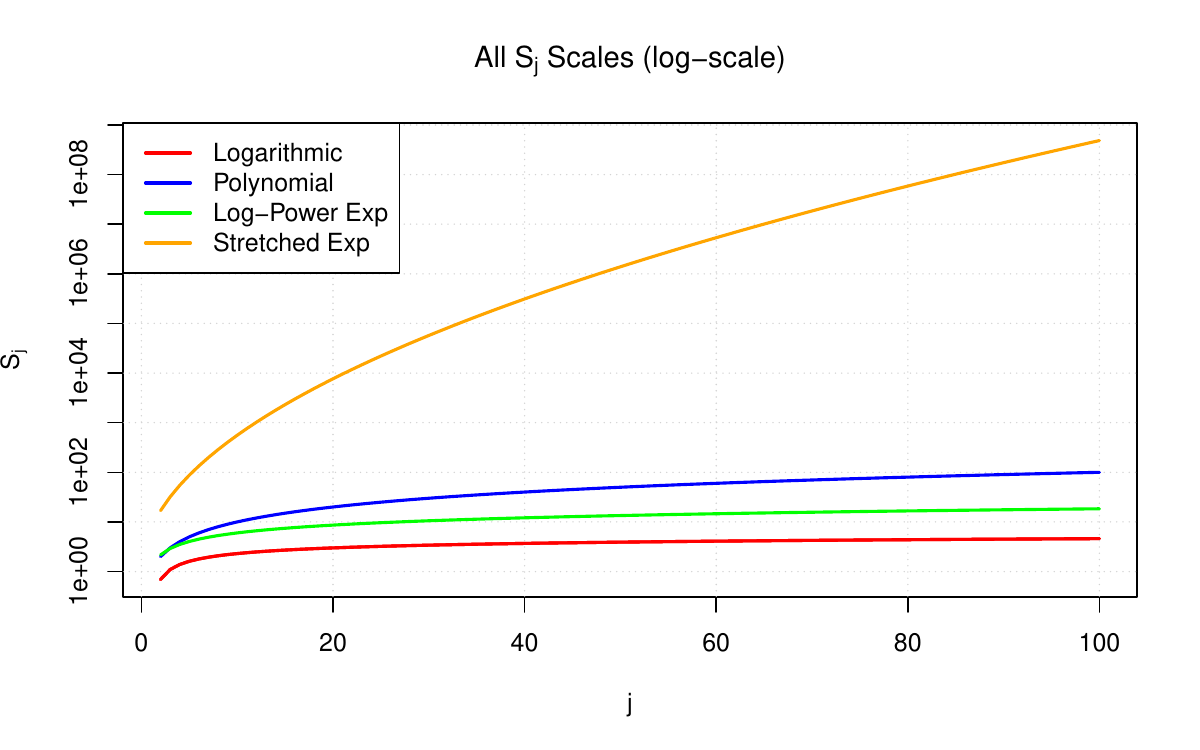}
		\caption{Growth of $S_j$ in log scale.}
		\label{fig:Sj}
	\end{subfigure}
	
	\caption{Comparison of shift $\varepsilon_j$ and scale $S_j$ across different asymptotic regimes. Each curve corresponds to a specific growth for the shifts
		%		\(\,b_j^2(u)=a_{j+1}(u)-a_j(u)\,\) for \(j\in\{2,5,8,10\}\), 
		%		where \(S_{j+1}=h_j\,S_j\).  
		%		(a) Total convergence case \(h_j=1+B^j\) with \(B=1.5\).  
				(a) Logarithmic case \(h_j=1+\eta/j\log j\) with \(\eta=1\).
				(b) Polynomial case \(h_j=1+\eta/j\) with \(\eta=1\).
				(c) Log-power exponential case \(h_j=1+\eta q (\log j)^{q-1}j^{-1}\) with \(q=0.7, \eta=1\).  
	 			(d) Mild exponential case \(h_j=1+j^{-p}\) with \(p=0.5\).  }
	\label{fig:combined}
	\end{figure}

%\begin{figure}[htbp]
%	\centering
%	\begin{subfigure}[t]{0.48\textwidth}
%		\includegraphics[width=\textwidth]{bj_plot_totconv.pdf}
%		\caption{\(h_j=1+1/j^p, p=2\).}
%	\end{subfigure}\hfill
%			\begin{subfigure}[t]{0.48\textwidth}
%		\includegraphics[width=\textwidth]{bj_plot_log.pdf}
%		\caption{\(h_j=1+\gamma/\log j, \gamma=1\).}
%	\end{subfigure}\hfill
%	\begin{subfigure}[t]{0.48\textwidth}
%		\includegraphics[width=\textwidth]{bj_plot_crit.pdf}
%		\caption{\(h_j=1+\eta/j,\ \eta=1\).}
%	\end{subfigure}
%	\begin{subfigure}[t]{0.48\textwidth}
%		\includegraphics[width=\textwidth]{bj_plot_stexp.pdf}
%		\caption{\(h_j=1+j^{-p},\ p=0.5\).}
%	\end{subfigure}\hfill
%	
%%	\vspace{2.5cm}
%	\caption{\textbf{Weight functions \(b_j(u)\) under four shift regimes in the shrinking case.}  
%		Each panel shows \(b_j(u)\) defined by 
%		\(\,b_j^2(u)=a_{j+1}(u)-a_j(u)\,\) for \(j\in\{2,5,8,10\}\), 
%		where \(S_{j+1}=h_j\,S_j\).  
%		(a) Total convergence case \(h_j=1+B^j\) with \(B=1.5\).  
%		(b) Logarithmic case \(h_j=1+\gamma/\log j\) with \(\eta=1\).
%		(c) Polynomial case \(h_j=1+\eta/j\) with \(\eta=1\).
%		(d) Mild exponential case \(h_j=1+j^{-p}\) with \(p=0.5\).  }
%	\label{fig:bj_shrink}
%\end{figure}

\begin{table}[h!]
	\centering
	\begin{tabular}{|c|c|c|c|}
		\hline
		\textbf{Regime} & $\varepsilon_j$&$S_j$  & $\Delta_j$ \\
		\hline
		\textit{Logarithmic} & $\frac{\eta}{j\log j}$ &$ \log ^{\eta} j $ & $ \frac{2\eta}{j \log j}$ \\
		\hline
		%\textit{Log-Polynomial} & $\frac{\eta}{j}+\frac{\gamma}{j \log j}$ & $ j^\eta \left(\log j \right)^\gamma $ & $\frac{2\eta}{j}$ \\
	%	\hline
		\textit{Polynomial} & $\frac{\eta}{j}$ & $ j^\eta $ & $\frac{2\eta}{j}$ \\
 		\hline
 		\textit{Log-Power Exponential} & $\eta q \frac{\left(\log j\right)^{q-1}}{j} $ &$\exp\left(\eta \left(\log j\right)^{q}\right)$ &  $2\eta q \frac{\left(\log j\right)^{q-1}}{j}  $ \\
 		\hline
		\textit{Mild Exponential} & $\eta j^{-p}$ &$\exp\left(\eta \frac{j^{1-p}}{1-p}\right)$ &  $2\eta\left(1-p\right)j^{-p}$ \\
		\hline
	\end{tabular}
		\vspace{0.5cm}
	\caption{Summary of dilation parameters in shrinking regimes as $j\rightarrow \infty$ (only leading terms).}
	\label{tab:shrink_regimes}
\end{table}

The next corollary, whose proof is in Section \ref{sec:proofs} provides a quantitative decay rate for shrinking needlets depending on the structure of the shifts $\lbrace \varepsilon_j : j \geq 1\rbrace$.
\begin{corollary}\label{cor:shrinkloc}
	In the shrinking regime, under the assumptions of Condition \ref{cond_epsilon} assuming $p \in \left[\left.-1,0\right)\right.$, the spatial localization bounds becomes as follows. For any $M \in \mathbb{N}$, there exists a positive constant $C_M$ such that
	\begin{equation*}
		\begin{split}
	\left \vert	\psi_{j,k} (x) \right \vert %& \leq C_M 4 \gamma(j) j^p S_{j}\frac{1}{\left(1+S_j\gamma(j)j^p\Theta_{j,k}\right)^{M}}\\
	%&	\simeq C_M 4 \gamma(j) j^p \exp\left( \frac{j^{p+1}\gamma(j)}{p+1}\right)\frac{1}{\left(1+\exp\left( \frac{j^{p+1}\gamma(j)}{p+1}\right)\gamma(j)j^p\Theta_{j,k}\right)^{M}} 
		&	\leq\frac{4 C_M  	\Sigma_{j;p}}{\left(1+	\Sigma_{j;p}\Theta_{j,k}\right)^{M}} ,\end{split}
	\end{equation*}  
	where
	\begin{equation}\label{eqn:locfunct}
		\Sigma_{j;p} = \begin{cases} \gamma(j) j^p \exp\left( \frac{j^{p+1}\gamma(j)}{p+1}\right) & p \in (-1,0)\\ 
			\gamma(j) \exp\left( \log j \gamma(j)\right) & p =-1, \quad \gamma(j) \geq O \left(\left( \log j\right)^{-1+\delta}\right)\\
			\eta \left(\log j \right)^{\eta-1}& p =-1, \quad \gamma(j) = \frac{\eta}{\log j}, \quad \eta>1\\
				\exp\left( j \gamma(j) \right)& p =0, \quad \gamma(j): \underset{j \rightarrow \infty}{\lim} \gamma(j)  =0.
			\end{cases}.
	\end{equation}
\end{corollary}
	The bound \eqref{eqn:locfunct} in Corollary \ref{cor:shrinkloc} captures a key feature of the shrinking regime. For any fixed angular distance $\Theta_{j,k}>0$, choosing $M$ arbitrarily large yields decay of order $\Sigma_{j;p}^{\,1-M}$, that is, faster than any inverse polynomial in $\Sigma_{j;p}$.  Thus, despite the shrinking frequency window at each step, the overall construction enjoys a still very powerful  localization and shrinking needlets are highly concentrated around their cubature nodes.\\
	As far as the total convergence is discussed, the spatial localization degrades. Indeed, the localization bound stabilizes to a constant, and loses the stretched‐exponential decay. Instead, for large $j$ 
	\[
	\left \vert \psi_{j,k}\right \vert \leq \frac{C_M}{\left(1+S_{\infty}\Theta_{j,k}\right)^M},
	\] 
	which decays only as a polynomial in the angular distance. In practical terms, when $p<-1$, the needlet construction achieves maximal spatial localization only up to a finite bandwidth, beyond which no further refinement is possible.\\
	As in the case of standard and Mexican needlets (see \cite{dur,npw2}), we determine the asymptotic behavior of the $L^p$-norms of flexible bandwidth needlets as follows.
\begin{corollary}\label{prop:norm} 
	For $q\in \left[ 1,+\infty \right] $,
	\begin{equation*}
		\left\Vert \psi _{j,k}\right\Vert _{L^{q}\left( \mathbb{S}^{2}\right)
		}=\left( \int_{\mathbb{S}^{2}}{\left\vert \psi _{j,k}\left( x\right)
			\right\vert ^{q}dx}\right) ^{\frac{1}{q}}\approx \Sigma_{j;p}^{2\left( \frac{1}{2}%
			-\frac{1}{q}\right) }%=\left( h_{j}S_{j}\right) ^{d\left( \frac{1}{2}-\frac{1%
				%}{p}\right) }
		.  \label{eq:pnorm}
		\end{equation*}
	\end{corollary}
	Unlike the standard case where the bounds depend primarily on the dyadic scaling $B^j$, here the estimate mirrors the specific functional structure of the sequence $S_j$, which governs the flexible localization in frequency. This result follows from the localization property and the tight frame structure of the flexible bandwidth needlet system. The proof is analogous to that for standard and Mexican needlets in \cite{dur,npw2} and is therefore omitted here.\\

\subsubsection{Some comments on the standard and the spreading cases}
Let us consider preliminarily the standard case, drawing some links between our results and the existing literature. The dilation factors approach here a constant $c>1$, and each new scale is a fixed as a multiple of the previous. Equivalently,
\begin{equation*}
	\varepsilon_j = h_j -1 \underset{j \rightarrow \infty}{\longrightarrow} c-1 >0.
\end{equation*}
It implies that for large $j$ the perturbations are essentially constant, that is, $\epsilon_j $ is slowly varying ($p=0$), with $\gamma$ such that $\lim_{j \rightarrow \infty} \gamma(j) = \eta$. %, for some $B>1$.
 In this case, the qualitative behavior of the scaling $S_j$ is exponential, since
\begin{equation*}
	\sum_{k=0}^{j-1}\varepsilon_j \sim \eta j,
\end{equation*}
and hence 
\begin{equation*}
	S_j = \exp \left(\sum_{k=0}^{j-1}\varepsilon_k\right) =\exp \eta j, 
\end{equation*}
which is pure exponential growth. In the special case where $h_j=B$, $B>1$, exactly for every $j$, one recovers the familiar geometric progression $S_j=B^j$.
Allowing small, bounded deviations of $\varepsilon_j$  around $\eta$ only alters lower‐order terms inside the exponential without changing the overall rate. %: the sequence remains exponentially growing at rate $c-1$. \\
To conclude it is important to stress that the relative bandwidth ratio case take the values 
\begin{equation*}
	\Delta_j = 2\sinh \left(\eta\right),
\end{equation*}
less tied to the scaling shift than the shrinking case. 
%Left panel of Figure \ref{fig:bj_spread} shows the construction of stable weight functions.
In the spreading regime the dilation factor itself grows without bound, that is,
\begin{equation*}
	\varepsilon_j = h_j -1 \rightarrow \infty,
\end{equation*}
so that the increment from one scale to the next becomes arbitrarily large. As a result, the partial sums
$\sum_{k=0}^{j-1} \varepsilon_k$ grow faster than linearly in $j$ yielding superexponential scaling $S_j$. 
If for example the shift is regularly varying with any $gamma$ and $p>0$, or with $p=0$ and diverging $\gamma$, we build the \textit{Stretched-super-exponential growth,} where $\varepsilon_j \sim j^{p}$, with $p>0$, then $\sum_{k=0}^{j-1} \varepsilon_k\sim \frac{j^{p+1}}{p+1}$ and 
\begin{equation*}
	S_j \sim \exp \left( \frac{j^{p+1}}{p+1} \right).
\end{equation*}  
Here, for the relative bandwidth ratio and the scale width we compute 
\begin{equation*}
		\Delta_j = 2 \exp\left(\frac{p\left(p-1\right)}{2}j^{p-2}\right) \sinh\left(pj^{p-1}\right)\left(1+o(j)\right).
\end{equation*}
It is important to remark that also rapidly varying shifts, that is, growing faster than any polynomials, lead significant spreading constructions, as for instance the \textit{double-exponential growth,} where $\varepsilon_j \sim aB^{j}$, with $a>0$ and $B>1$, and $\sum_{k=0}^{j-1} \varepsilon_k\sim a\frac{B^j-1}{B-1}$, while
\begin{equation*}
	S_j \sim \exp \left( \frac{a\left(B^{j}-1\right)}{\left(B-1\right)} \right).
\end{equation*} 
In needlet constructions this type of regime produces nearly disjoint support windows $\left[S_{j-1},S_{j+1}\right]$ at high levels, which can undermine both frame overlap and partition of unity properties unless the weight functions are carefully adjusted. Concerning the relative bandwidth ratio and the scale width we obtain
\begin{equation*}
\Delta_j = \exp\left(aB^j\right) \left(1+o\left(j\right)\right), 
\end{equation*}
exploding as $j$ grows to infinity.

\section{Scale Regimes for Optimal Correlation Decay}\label{sec:corr}
We now turn our attention to the high-frequency uncorrelation properties of needlet coefficients. Specifically, we investigate the behavior of correlations in the spherical random field described by equation \eqref{eqn:randomfield}, focusing on their evaluation through flexible-bandwidth needlet coefficients. 
We focus on needlets in the shrinking regime, corresponding to a setting where the bandwidth ratio $\Delta_j$ converges to 1 as $j\rightarrow \infty$. implying that the support windows for $b_j$ become increasingly narrow at high frequencies.
This regime is particularly advantageous when a sharper asymptotic control is needed. Indeed, narrower bands lead to tighter localization in harmonic space, which often translates into improved error bounds or convergence rates in high-frequency asymptotics. Also, This approach proves especially effective in situations where sparsity or compressibility is present. If the underlying signal or field exhibits rapid decay or is concentrated in specific frequency regions, shrinking bands can isolate and capture informative components more efficiently.
In the present work, as well as in more general statistical settings, shrinking bands help isolate finer scales, reducing cross-scale interference and allowing for better variance control and bias reduction in central limit theorems or estimation problems. Finally, his framework is especially advantageous when high-resolution data is available. When observing data at very fine angular resolution (e.g., cosmology, geophysics), the shrinking regime provides a natural multiscale refinement that can better exploit the structure of the signal.
In this paper, we relate previous results on asymptotic uncorrelation of needlet coefficients to the different regimes in the shrinking cases within the discrete needlet constructions with variable bandwidth. We analyze the decay of correlations across different frequency bands, under the regimes we have introduced, and provide a detailed asymptotic study of their behavior. This includes a comparison across polynomial and logarithmic growth scenarios for bandwidth functions, allowing us to evaluate how different choices impact the decay of correlations and the degree of localization in both spatial and frequency domains. Our results offer a comprehensive characterization of high-frequency decorrelation, tailored to a wide range of power spectrum behaviors and adaptable to practical needs in data analysis on the sphere. 
\subsection{Some background results on uncorrelation between needlet coefficients}
We operate here under the assumption that the angular power spectrum satisfies the regularity condition \ref{eqn:regularity}. Such scale-invariant behavior is a common assumption in theoretical models, particularly in idealized cosmological settings. However, in practical applications—such as the analysis of the Cosmic Microwave Background (CMB)—the angular power spectrum often displays more intricate structure, including oscillatory features superimposed on a decaying envelope (see, e.g., \cite{dmt24,planck1}). These oscillations arise from physical phenomena such as acoustic peaks in the early universe and can complicate the statistical behavior of needlet coefficients.\\
Under the assumptions in Condition \ref{cond:Cl}, it has been proved in \cite{dmt24} that for any $N \in \mathbb{N}$, as $j\rightarrow \infty$,%
\begin{equation*}
	\left \vert \operatorname{Corr} \left(\beta_{j,k},\beta_{j,k^\prime} \right)\right \vert \leq c_N \max \left(\frac{1}{\left(S_{j-1}^{1-\beta}\Theta_{j,k,k^\prime}\right)^{N}},\frac{1}{\left(\left(S_j-S_{j-1}\right)\Theta_{j,k,k^\prime}\right)^{N}}\right),
\end{equation*}
where $c_N$ depends only on $\alpha$ and $G$, but not on $j$, and where $\Theta_{j,k,k^\prime}$ is the great circle distance between $\xi_{j,k}$ and $\xi_{j,k^{\prime}}$. Note that the minimum distance between cubature points as the resolution level $j \rightarrow \infty$ goes as 
\begin{equation}\label{eqn:dS}
	d_j = \min_{k\neq k^{\prime}=1, \ldots,K_j} \Theta_{j,k,k^\prime} \simeq S_{j-1}^{-1}. 
\end{equation}
see again \cite{dmt24}. 
 The relation \eqref{eqn:dS} means that the angular distance $\Theta_{j;k,k^{\prime}}$ between two needlet centers $\xi_{j,k}$ and $\xi_{j,k^{\prime}}$ scales like the inverse of the scale of the previous window support. 
This scaling reflects how needlets are approximately localized at a resolution controlled by $\min\left(S_{j-1}, S_{j}-S_{j-1}\right)$. As $j\rightarrow \infty$, both diverge so the spatial resolution improves and, consequently, the centers $\xi_{j,k}$ and $\xi_{j,k^\prime}$ are placed more densely, with separation of order $S_{j-1}$. Despite the window support being $\left[S_{j-1},S_{j+1}\right]$, the angular resolution is governed by the coarsest scale in the band $S_{j-1}$, even though higher frequencies are present up to $S_{j+1}$.\\
In \cite{dmt24}, the authors exploit nearly‐uncorrelated needlet coefficients to build a powerful goodness–of–fit test for Gaussian isotropic fields on the sphere. The idea is as follows: at each resolution level $j$, only a subsample 
$D_j$ of cubature points is selected, so that any two sampled points are separated by at least $O\left( S_{j-1}^{1-\beta-\epsilon}\right)$, $\epsilon>0$. Under the condition $S_{j-1}^{1-\beta} \leq S_{j}-S_{j-1}$, the needlet coefficients $\beta_{j,k}$ and $\beta_{j,k^\prime}$ satisfy
\begin{equation*}
\operatorname{Corr} \left(\beta_{j,k}\beta_{j,k^\prime}\right)
\leq C_M  S_{j-1}^{1-\beta} \Theta_{j;k,k^\prime}^{-M}
\end{equation*}
so that, on the one hand, the variance of the (normalized) test statistic 
$$I_j = \frac{1}{\operatorname{card}\left(D_j\right)}
\sum_{k \in D_j}\left( \beta_{j,k}^2 -1\right)$$ converges to 1, and, on the other hand, its fourth cumulant vanishes at rate $O\left(\operatorname{card}\left(D_j\right)^{-1}\right)$. By the Malliavin–Stein method this yields a quantitative Central Limit Theorem for 
$I_j$, and hence a valid, high‐resolution goodness–of–fit procedure.
\subsection{Needlet uncorrelation and shrinking growth}
In this section we establish which shrinking regime choices of scaling centers $\lbrace S_j: j \geq 1\rbrace$ satisfy that lower bound on 
$S_j-S_{j-1}$, thereby guaranteeing the uncorrelation rates needed for these tests to remain effective when the needlet bandwidth varies, under the assumptions regular variations on the dilation parameter by means of its shift.\\
 This asymptotic regime guarantees sufficient concentration of the needlet window functions in both pixel and frequency domains, while preserving a non-trivial geometric separation between adjacent frequency bands. %Such a balance is particularly beneficial for the theoretical analysis and practical implementation of the goodness-of-fit test constructed from needlet coefficients, as discussed in \cite{dmt24} and presented above.
More in detail, if Condition \eqref{cond_epsilon} holds, the growth of the scaling center is controlled by the regularity index $p$. In what follows, we will examine for which choices of $p$ we have that 
\begin{equation}
%\delta_j\left(\beta\right) = 
R_{j}\left(\beta\right)=\frac{S_j-S_{j-1}}{S_{j-1}^{1-\beta}} > 1, \label{eqn:corrbeta}
\end{equation}
so that denominator compensates for overall scale growth at the numeratori in a way that depends on the smoothness index $\beta$. 
Indeed, to achieve both sharp frequency localization (minimizing bias) and rapid decay of coefficient correlations (controlling variance), it is crucial that successive needlet bands separate at a rate tuned to the power‐spectrum smoothness.  Concretely, imposing $S_j^{1-\beta} < S_j - S_{j-1}$ as  $j \to \infty $ ensures that the “one‐step” increment in scale outpaces the natural smoothness scale \(S_j^{1-\beta}\). This condition yields optimal spatial decorrelation bounds while preserving narrow bandwidths, which directly benefits the power and reliability of needlet‐based goodness‐of‐fit procedures.
% We seek regimes in which
% \[
% \Delta_j = \frac{S_{j+1}-S_{j-1}}{S_j} \;\longrightarrow\;0
% \quad\text{and}\quad
% S_j^{\,1-\beta} < S_j - S_{j-1}
% \quad(j\to\infty).
% \]
\begin{theorem}\label{thm:corr}
Under the assumptions of Conditions \ref{cond:Cl} and \ref{cond_epsilon}, where 
\[
S_j = \begin{cases} \exp\left(\gamma(j)j\right) & p=0, \lim_{j \rightarrow0} \gamma(j) = 0\\
	\exp\left(\gamma(j)\frac{j^{p+1}}{p+1}\right) & p\in(-1,0),\\
j^{\gamma(j)}& p=-1, \gamma(j) = O\left(\left(\log j\right)^{-1+\delta}\right)\\
\left( \log j\right)^\eta& p=-1, \gamma(j) = \frac{\eta}{\log j}(1 +o(j)),
	\end{cases}
\]
with $\gamma$ slowly varying,
Condition \eqref{eqn:corrbeta} is satisfied by the following regimes within the shrinking framework in the divergent subexponential growth:
	\begin{itemize}
		\item$ p\in(-1,0]$; 
		\item $p = -1 $ and $\lim_{j \rightarrow \infty} \gamma(j) = \gamma_{\infty} > \frac{1}{\beta}$.
	\end{itemize}
\end{theorem}
This theorem implies that, when studying random fields satisfying condition \ref{cond:Cl}, the optimal choices for the dilation sequence exclude most of those associated with the critical regularity value $p=-1$. Instead, it favors sequences $_j$ that continue to grow sufficiently fast to capture the terms that depend on $\beta$. In particular, this result rules out dilation sequences with slower growth rates, such as logarithmic growth and certain polynomial growths where the limiting exponent $\gamma_{\infty}$ satisfies $\gamma_{\infty}<1/\beta$. These slower-growing sequences shrink "too fast" and therefore fail to detect the relevant $\beta$-dependent asymptotic behavior. This distinction highlights the delicate balance required in choosing growth regimes that are neither too slow nor too fast to ensure effective detection of the underlying regularity features of the random fields under consideration.  
\section{Proofs}\label{sec:proofs}
This section contains the proofs of the results presented throughout the paper.
\begin{proof}[Theorem \ref{thm:regime}] 
	To prove our main results, we will relate relative bandwidth ratios and the characterizations of the growth classes in terms of the sequence 
	$\lbrace L_j: j \geq 1 \rbrace$, distinguishing between subexponential, exponential, and superexponential behavior. We treat each regime separately in the following analysis.
	\paragraph{Shrinking regime ($L_j \rightarrow 0  \Leftrightarrow S_j \rightarrow 0$).} 
	First, note that if $S_{j+1}>S_j$, using $S_j = \exp\left(jL_j\right)$, we have that
	\begin{equation}\label{eqn:DeltaL}
		\Delta_j %= \exp\left(\left(j+1\right)L_{j+1}-j L_j\right)- \exp\left(\left(j-1\right)L_{j-1}-j L_j\right) 
		=\exp(A_j) -\exp(-A_{j-1}), %= \frac{L_{j+1}-L_{j-1}}{L_{j}} + \frac{L_{j+1}-L_{j-1}}{jL_{j}} .
	\end{equation}
	where $	A_j = \left(j+1\right)L_{j+1}-j L_j$
	If we assume $L_j \rightarrow 0$, then for any $\delta>0$, there exists $J_\delta>0$ such that $\left \vert L_{j} \right \vert<\delta_1$ and we want to show that for every $\epsilon_1>0$, there exists $J_{\epsilon_1}$ such that for all $j\geq J_{\epsilon_1}$, $\Delta_j \leq \epsilon_1$.  From the assumption, for $j >J_{\delta_1}$, it holds that 
	\begin{equation*}
		\left \vert L_{j+1}-L{j} \right \vert \leq 2\delta_1, \quad	\left \vert L_{j-1}-L{j} \right \vert \leq 2\delta_1
	\end{equation*}
	and thus 
	\begin{equation*}
		\begin{split}
			\left \vert A_j +A_{j-1} \right \vert & =   \left \vert (j+1) L_{j+1} -(j-1) L_{j-1} \right \vert\\
			& \leq  (j+1  )\left \vert L_{j+1} \right \vert +    (j-1  )\left \vert L_{j-1} \right \vert \\
			& \leq 2\delta_1 j.
		\end{split}
	\end{equation*}
	Expanding Equation \eqref{eqn:DeltaL} yields 
	\begin{equation*}
		\begin{split}
			\left \vert \Delta_j \right \vert% &\leq \left \vert A_j - B_j \right \vert + O\left(A_j^2+B_j^2\right) \\
			& \leq 2j \delta_1 + \delta_1^2 o(j).
		\end{split}
	\end{equation*}
	For fixed $\epsilon_1>0$, we pick $\delta_1=\frac{\epsilon_1}{2 j}$, such that if for sufficiently large $, j$ $\left \vert L_j \right 	\vert\leq \epsilon_1 / 4j$ and 
	\begin{equation*}
		\left \vert \Delta_j \right \vert \leq \epsilon_1.
	\end{equation*} 
	Conversely, suppose for the sake of contradiction that  $\underset{j \rightarrow \infty}{\lim}\Delta_j =0$, but $\underset{j \rightarrow \infty}{\lim}L_j >0$. Then there exists  $\eta>0$ and an infinite sequence of indices $\lbrace j_q : q \in \mathbb{N}\rbrace$, where $j_1<j_2<\ldots$, such that
	$L_{j_q}\geq \eta$ for every  $q \in \mathbb{N}$. Thus,
	\begin{equation*}
		S_{j_q} \geq \exp\left(\eta j_q\right) = \epsilon_2^{j_q},
	\end{equation*}
	where $\epsilon_2=e^{\eta}$
	On the other hand, we know that for any $\delta_2>0$, there exists $J_{\delta_2}>0$ such that for $j \geq J_{\delta_2}$, $\Delta_j \leq \delta_2$ and thus 
	\begin{equation*}
		S_{j+1}-S_{j-1} \leq \delta_2 S_j. 
	\end{equation*}
	It implies that 
	\begin{equation*}
		S_{j+1}-S_{j} \leq S_{j+1}-S_{j-1}\leq \delta_2 S_j \rightarrow S_{j+1}\leq(1+\delta_2)S_j,
	\end{equation*}
	while
	\begin{equation*}
		S_{j}-S_{j-1} \leq S_{j+1}-S_{j-1}\leq \delta_2 S_j \rightarrow S_{j-1}\geq(1-\delta_2)S_j,
	\end{equation*}
	Iterating, we have that $S_j \leq S_{j_0} (1+\delta_2)^{j-j_0}$. We choose $j_0>J_{\delta_2}$ and we obtain
	\begin{equation*}
		\begin{split}
			L_j &= \frac{\log S_j}{j} \\
			& \leq \frac{\log S_{j_0}}{j} + \left(1-\frac{j_0}{j}\right)\log(1+\delta_2),
		\end{split}
	\end{equation*}
	thus 
	\begin{equation*}
		\underset{j \rightarrow \infty}{\limsup} L_j \leq \log \left(1+\delta_2\right). 
	\end{equation*}	
	Combining the two inequalities, we have
	\begin{equation*}
		\epsilon_2^{j_q}\leq S_{j_q} \leq S_{j_0} (1+\delta_2)^{j_q-j_0}
	\end{equation*} 
	yielding
	\begin{equation*}
		\left(\frac{\epsilon_2}{1+\delta_2}\right)^{j_q}\leq S_{j_q} \leq S_{j_0} (1+\delta_2)^{-j_0},
	\end{equation*} 
	where $S_{j_0} (1+\delta_2)^{-j_0}$ is a finite constant. This is true only if $\frac{\epsilon_2}{1+\delta_2}< 1$, but since $\delta_2$ can be chosen arbitrarily small, we reach a contradiction.
	\paragraph{Standard regime ($L_j \rightarrow c  \Leftrightarrow L_j \rightarrow c^\prime$).} 
	We assume that $c^\prime = \log B$, for $B>1$ and then $L_j \rightarrow \log B$. Hence, $ S_j \approx B^j$ and it is easy to note that 
	\begin{equation*}
		\Delta_j \rightarrow B-\frac{1}{B}>0. 
	\end{equation*} 
	Conversely, assuming that $ \Delta_j \rightarrow c$, we note that
	\begin{equation*}
		\Delta_j = \kappa_j - \frac{1}{\kappa_{j-1}}
	\end{equation*}
	where $\kappa_j =S_{j+1}/S_j$ to be bounded away from the unity in the limit.This implies that there exists $\epsilon_3 >0$ and $J_{\epsilon_3}$ such that for $j \geq J_{\epsilon_3}$, such that 
	$S_{j+1}/S_{j} \geq 1+\epsilon_3$. Thus we obtain recorsively that $S_j \geq S_{j_0}(1+\epsilon_3)^{j-j_0}$, with $j_0 >J_{\epsilon_3}$. Thus
	\begin{equation*}
		S_j \geq \delta_3 (1+\epsilon_3)^{j},
	\end{equation*}
	with $\delta_3 = S_{j_0}(1+\epsilon_3)^{-j_0}>0$, as claimed. Analogously, we can prove that $\Delta_j$ has to be bounded above, that is $\kappa_j+\kappa_{j-1}^{-1}< \epsilon_4$. This allows us to gain the same recursion formula leading to an analogous bound of the type
	\begin{equation*}
		S_j \leq \delta_4 (1+\epsilon_4^\prime)^{j}.
	\end{equation*}
	Here we omit the details for the sake of the brevity. 
	\paragraph{Spreading regime ($L_j \rightarrow \infty  \Leftrightarrow L_j \rightarrow \infty$).} 
	The assumption $ L_j \rightarrow \infty$ implies that for any $\delta_5>0$, there exists $J_{\delta_5}$ such that for any $j \geq J_{\delta_5}$,
	\begin{equation*}
		S_j >e^{\delta_5 j}.
	\end{equation*}
	Fixing $\delta_5=\epsilon_5^{-1}$, we obtain
	\begin{equation*}
		\frac{S_{j+1}}{S_j} > e^{\frac{1}{\epsilon_5}}; \quad   		\frac{S_{j-1}}{S_j} < e^{-\frac{1}{\epsilon_5}}.
	\end{equation*}
	It follows that 
	\begin{equation*}
		\Delta_j	\leq e^{\frac{1}{\epsilon_5}}+e^{-\frac{1}{\epsilon_5}} \rightarrow \infty, 
	\end{equation*}
	since $\underset{j \rightarrow \infty }{e^{\frac{1}{\epsilon_5}}}=\infty$, and $\underset{j \rightarrow \infty }{e^{-\frac{1}{\epsilon_5}}}=0$.\\
	Assume now that $ \Delta_j \rightarrow \infty$, that is, for any $\delta_6>0$, there exists $J_{\delta_6}>0$ such that for $j >J_{\delta_6}$, 
	\begin{equation*}
		\frac{S_{j+1}-S_{j-1}}{S_j} > \delta_6. 
	\end{equation*} 
	Set $\epsilon_6>0$; we fix $\delta_6= e^{\epsilon_6}$. Then, $S_{j+1}-S_{j-1}>e^{\epsilon_6}S_j$. Since $S_{j-1}>0$, then $S_{j+1}>e^{\epsilon_6}S_j$. Now, iterating this inequality and using $S_{J_{\epsilon_6}}\geq1$ 
	\begin{equation*}
		S_{{\epsilon_6}+q}\geq e^{q{\epsilon_6}}, \quad \text{ for }q\geq 0.
	\end{equation*}
	Fixing $j=J_{\epsilon_6}+q$ It implies that 
	\begin{equation*}
		S_{j}\geq e^{-J_{\epsilon_6}{\epsilon_6}}\left(e^{{\epsilon_6}}\right)^j, \quad \text{ for }q\geq 0.
	\end{equation*}
	Since ${\epsilon_6}$ can be chosen  arbitrarily, it follows that 
	\begin{equation*}
		L_j \geq {\epsilon_6} +o(j)\rightarrow {\epsilon_6}, 
	\end{equation*}
	which completes the argument.
\end{proof}
 \begin{proof}[Theorem \ref{thm:local}]
	Let us first focus on the shrinking regime. We have to show that $\underset{j\rightarrow \infty}{\lim}\Delta_j=0$, implies $S_{j}-S_{j-1}< S_j$. 
	Observe that 
	\begin{equation}\label{eqn:corprof}
		\Delta_j = \frac{S_{j+1}-S_{j}}{S_{j}} + \frac{S_{j}-S_{j-1}}{S_{j}}.	
	\end{equation} 	
	Both the addends are positive and since $\Delta_j \rightarrow 0$, thus $\frac{S_{j}-S_{j-1}}{S_{j}}$ converges to $0$. Thus, for any $\epsilon>0$, there exists $J_{\epsilon}>0$ such that for any $j \geq J_{\epsilon}$, 
	\begin{equation*}
		\frac{S_{j}-S_{j-1}}{S_{j}} <\epsilon.	
	\end{equation*} 	
	Setting $\epsilon=1/2$, we have that for sufficiently large $j$, $S_{j}-S_{j-1}<S_{j}/2$. On the other hand, we have also that
	\begin{equation*}
		\frac{S_{j-1}}{S_{j}} = 1 - \left(\frac{S_{j}-S_j-1}{S_j}\right)>\frac{1}{2},
	\end{equation*}
	when $\frac{S_{j}-S_j-1}{S_j}<\frac{1}{2}$, so that $S_{j-1}>\frac{S_j}{2}$. Combining the two inequalities proves our claim.\\
	Conversely, if $\underset{j\rightarrow \infty}{\lim}\Delta_j=\infty$, implies $S_{j}-S_{j-1}> S_j$. Consider Equation \eqref{eqn:corprof} and observe that $0 <\frac{S_{j}-S_{j-1}}{S_{j}}\leq 1$, and thus if $\Delta_j \rightarrow \infty$, necessarily $\underset{j\rightarrow \infty}{\lim}\frac{S_{j+1}-S_j}{S_j}=\infty$. Hence, for any $\delta>0$, there exists $J_\delta>0$ such that for any $j > J_{\delta}$,
	\begin{equation*}
		\frac{S_{j+1}-S_{j}}{S_{j}} >\delta.	
	\end{equation*} 	 
	Setting $\delta=1$ implies that for sufficiently large $j$, $S_{j+1}-S_{j}>S_j$. Rescaling the index completes this part of the proof. \\
	To conclude the proof, as already mentioned by \cite{dmt24}, just notice that in the standard regime, $S_j \approx B^j$, so that, while $S_{j-1}= B^j /B$, $S_{j}-S_{j-1}=B^{j}(1-1/B)$, that is, $S_{j}-S_{j-1}\approx S_{j-1}$, as claimed.
\end{proof}
	\begin{proof}[Theorem \ref{thm:regimes}] To prove the relations between regular varying shifts with and the related regimes, we first observe that under the assumptions of Condition \ref{cond_epsilon}, it holds that
	\[
	\sum_{k=1}^{j-1} \varepsilon_k \underset{j \rightarrow\infty}{\longrightarrow}\begin{cases} 
		\left(\gamma(j) \frac{j^{p+1}}{p+1}\right) & \text{ for }p\neq -1,\\
		\left(\gamma(j) \log j \right)& \text{ for }p = -1, \gamma(j) \neq O\left(\left(\log j\right)^{-1}\right),\\
		\eta \log\left( \log j \right)& \text{ for }p = -1, \gamma(j) =\eta \left(\log j\right)^{-1}, \eta>0.
	\end{cases}
	\]
Noting that $\log S_j = 	\sum_{k=1}^{j-1} \varepsilon_k $, it follows easily that for $p=-1$, $L_j =\log S_j /j$ converges to 0 as $j \rightarrow \infty$, while for $p\neq -1$ we have that
		\[
		L_j = \frac{\log S_j}{j}= \gamma(p)\frac{j^p}{1+p}.
		\]
The index $L_j$ diverges for $p>0$, yielding the spreading regime, and converges to 0 for $p<0$, leading to the shrinking case. When $p=0$, the asymptotic behavior of $L_j$ depends only on $\gamma$, and we can obtain shrinking ($\gamma \rightarrow 0$), stable ($\gamma \rightarrow c$), and spreading ($\gamma \rightarrow \infty$), as claimed.\\
Let us focus on the shinking regime and study the total convergence setting. observe that if $p<-1$, 
	\begin{equation*}
		\sum_{k=1}^{j-1} \varepsilon_k 	\underset{j \rightarrow \infty}{\rightarrow}E, 
	\end{equation*}
	as claimed. Moreover, if $p=-1$, and $\gamma(j)=O\left(\left(\log j\right)^{-1-\delta}\right)$, $\delta>0$, we have that 
		\begin{equation*}
		\sum_{k=1}^{j-1} \varepsilon_k 	\underset{j \rightarrow \infty}{\rightarrow} \log j \gamma(j) \rightarrow 0,
	\end{equation*}
	and thus, in both the cases, $S_j$ converges to a finite constant.
	Concerning the divergent regularly varying dilation, it is straightforward that, for $p \in (-1,0)$,
	\begin{equation*}
		\sum_{k=0}^{j}\varepsilon_k \simeq= \frac{\gamma(j)}{1+p} j^{1+p}\rightarrow \infty,
	\end{equation*} 
	as claimed. In the case $p=0$ and $\gamma(j):\lim_{j \rightarrow \infty} \gamma (j) = 0$, we find
	\[
	\sum_{k=1}^{j-1} \varepsilon_k 	\underset{j \rightarrow \infty}{\rightarrow}  j \gamma(j) \rightarrow \infty.
	\]
	For $p=-1$, and $\gamma(j) = O\left(\left(\log j\right)^{-1+\delta}\right)$,
	we note that 
		\[
	\sum_{k=1}^{j-1} \varepsilon_k 	\underset{j \rightarrow \infty}{\rightarrow}  \log j \gamma(j) \rightarrow \infty,
	\]
	while for the boundary case $p=-1$, and $\gamma(j) =\frac{\eta}{\log j}$,
	\[
\sum_{k=1}^{j-1} \varepsilon_k 	\underset{j \rightarrow \infty}{\rightarrow} \eta \log \log j \rightarrow \infty,
\]
as claimed.
\end{proof}
\begin{proof}[Corollary \ref{cor:shrinkloc}]
	Recall that if Condition \ref{cond_epsilon} holds and if  $p \in \left[\left.-1,0\right)\right.$, using
	$\varepsilon_j = \gamma( j)^{p}$ with $S_j \simeq \exp\left(\sum_{k=0}^{j-1}\varepsilon_j \right)$ yields
	\[
	\frac{S_{j+1}^2}{S_{j-1}^2} = \exp\left(2\left(\varepsilon_{j-1}+\varepsilon_{j}\right)\right).
	\]
	Thus
	\[
	\frac{S_{j+1}^2-S_{j-1}^2}{S_{j-1}} =S_{j-1}\left(\exp\left(2\left(\varepsilon_{j-1}+\varepsilon_{j}\right)\right)-1\right).
	\]
	Since $\lim_{j \rightarrow \infty} \varepsilon_j =0$, we can expand the exponential term into 
	\[
	\left(\exp\left(2\left(\varepsilon_{j-1}+\varepsilon_{j}\right)\right)-1\right)  = 2\left(\varepsilon_{j-1}+\varepsilon_{j}\right) + 2 2\left(\varepsilon_{j-1}+\varepsilon_{j}\right)^2 + \ldots,
	\]
	so that 
	Thus
	\[
	\frac{S_{j+1}^2-S_{j-1}^2}{S_{j-1}^2} =
	2S_{j-1} \left(\varepsilon_{j-1}+\varepsilon_{j}\right)+O\left(S_{j-1} \left(\varepsilon_{j-1}+\varepsilon_{j}\right)\right).
	\]
	We are now ready to use the explict definition of $\varepsilon$, recalling that $\gamma$ is slowly varying and that
	\[
	\left(j-1\right)^p = j^{p}\left(1-\frac{p}{j}+o\left(\frac{1}{j}\right)\right),
	\]
	to get
	\[
	\varepsilon_{j-1}+\varepsilon_{j} = \gamma(j)j^p \left(2-\frac{p}{j}+o\left(\frac{1}{j}\right)\right).
	\]
	Thus 
	\[
	\frac{S_{j+1}^2-S_{j-1}^2}{S_{j-1}} = 4 \gamma(j) j^p S_{j-1} \left(1 + O\left(\frac{1}{j}\right)\right).
	\]
	Analogously, we have that
	\[S_j-S_{j-1} =S_{j-1}\left(\exp\left(\varepsilon_j\right)-1\right),\]
	and a similar reasoning yields
	\[
	S_j-S_{j-1} \simeq S_{j-1}\gamma(j)j^p.
	\]
	Since asymptotically $S_{j-1} =S_j\left(1+O\left(\varepsilon_{j-1}\right)\right)$, we obtain the claimed result.
\end{proof}
\begin{proof}[Proof of Theorem \ref{thm:corr}]
	Recall that under the assumptions of Condition \ref{cond_epsilon}, $\varepsilon_j = \gamma(j)j^p$, where $\gamma$ is slowly varying and $p \in \left[\left.-1,0\right)\right.$. Since $	S_j = \exp\left(\sum_{k=1}^{j-1} \varepsilon_k\right)$, using the Taylor expansion yields
	\[
	S_j-S_{j-1}\simeq	\varepsilon_j S_{j-1} ,
	\] 
	and thus
	\[
	R_j(\beta) = \varepsilon_{j} \exp\left(\beta \sum_{k=1}^{j-2}\varepsilon_k\right).
	\]
	Then it holds that:
	\[
	R_j(\beta) >1 \Leftrightarrow \beta \sum_{k=1}^{j-2}\varepsilon_k > -p \log j - \log \gamma(j).
	\]	
	If $p\in(-1,0]$,  $\sum_{k=1}^{j-2}\varepsilon_k \simeq \frac{j^{1+p}\gamma(j)}{1+p}$ and then 
	\begin{align*}
		\lim_{j \to \infty} \frac{-p \log j - \log L(j)}{j^{1+p}\gamma(j)} = 0.
	\end{align*}
	For any $\beta > 0$, we have $R_j(\beta) > 1$ eventually. \\
	If $p=-1$, we have that $\sum_{k=1}^{j-2}\varepsilon_k \simeq \gamma(j)\log(j)$. Hence 
	\[
	\frac{\log j- \log \gamma(j)}{\gamma(j)\log j} =  \frac{1}{\gamma(j)}  -\frac{\log \gamma (j)}{\gamma(j) \log j }.
	\]
	Now, it holds that
	\begin{itemize}
		\item If $\gamma(j) \to \gamma_\infty > 0$, then $R_j(\beta) > 1$ eventually if and only if  $\beta > 1/\gamma_\infty$.
		\item If $\gamma(j) \to \infty$, then $R_j(\beta) > 1$ for any $\beta > 0$.
		\item If $\gamma(j) \to 0$, then $R_j(\beta) \to 0$ for all $\beta < 1$,
		as claimed.
	\end{itemize}
\end{proof}
	\section*{Funding}
This work was partially supported by the PRIN 2022 project GRAFIA (Geometry of Random Fields and its Applications), funded by the Italian Ministry of University and Research (MUR).

\end{document}